\documentclass{amsart}
 
\usepackage{a4wide}
\usepackage{amsfonts}
\usepackage{amsthm}
\usepackage[english]{babel}
\usepackage{bbm}
\usepackage{hyperref}
\usepackage{xcolor}

\newcommand{\E}{\mathbb E}
\newcommand{\N}{\mathbb N}
\renewcommand{\P}{\mathbb P}
\newcommand{\R}{\mathbb R}
\newcommand{\Z}{\mathbb Z}
\newcommand{\half}{\mbox{$\frac 1 2$}}

\renewcommand{\div}{\operatorname{div}}
\newcommand{\grad}{\operatorname{grad}}
\newcommand{\1}{\mathbbm 1}
\newcommand{\im}{\operatorname{im}}
\newcommand{\deRham}{\operatorname{deRham}}

\newtheorem{theorem}{Theorem}[section]
\newtheorem{lemma}[theorem]{Lemma}
\newtheorem{proposition}[theorem]{Proposition}

\newtheorem{corollary}[theorem]{Corollary}

\theoremstyle{definition}
\newtheorem{definition}[theorem]{Definition}

\newtheorem{problem}[theorem]{Problem}

\theoremstyle{remark}
\newtheorem{remark}[theorem]{Remark}
\newtheorem{example}[theorem]{Example}

\begin{document}

\title[Ergodic control on compact manifolds]{Linear PDEs and eigenvalue problems corresponding to ergodic stochastic optimization problems on compact manifolds}
\author{Joris Bierkens}
\address{J. Bierkens \\
University of Warwick \\
Department of Statistics \\
Coventry, CV4 7AL \\
UK}
\email{j.bierkens@warwick.ac.uk}
\thanks{The research at Radboud University by J. Bierkens and H. J. Kappen has received funding from the European Community's Seventh Framework Programme (FP7/2007-2013) under grant agreement no. 270327 (CompLACS).}

\thanks{The research at the University of Warwick by J. Bierkens has received support from the EPSRC under the CRiSM grant: EP/D002060/1.}

\author{Vladimir Y. Chernyak}
\address{V. Y. Chernyak \\ Wayne State University \\ Department of Chemistry \\ Detroit (MI), USA}
\email{chernyak@chem.wayne.edu}
\thanks{The research at Wayne State University (MI), USA, has received support from the NSF under grant agreement no. CHE-1111350.}

\author{Michael Chertkov}
\address{M. Chertkov \\ Los Alamos National Laboratory \\ Center for Nonlinear Studies \\ Los Alamos (NM), USA}
\email{chertkov@lanl.gov}
\thanks{The work at LANL was carried out under the auspices of the National Nuclear Security Administration of the U.S. Department of Energy at Los Alamos National Laboratory under Contract
No. DE-AC52-06NA25396.}

\author{Hilbert J. Kappen}
\address{H. J. Kappen \\ Radboud University \\ Faculty of Science \\ Nijmegen \\ The Netherlands}
\email{b.kappen@science.ru.nl}

\begin{abstract}
Long term average or `ergodic' optimal control problems on a compact manifold are considered. The problems exhibit a special structure which is typical of control problems related to large deviations theory: Control is exerted in all directions and the control costs are proportional to the square of the norm of the control field with respect to the metric induced by the noise.
The long term stochastic dynamics on the manifold will be completely characterized by the long term density $\rho$ and the long term current density $J$. As such, control problems may be reformulated as variational problems over $\rho$ and $J$. The density $\rho$ is paired in the cost functional with a state dependent cost function $V$, and the current density $J$ is paired with a vector potential or gauge field $A$.
We discuss several optimization problems: the problem in which both $\rho$ and $J$ are varied freely, the problem in which $\rho$ is fixed and the one in which $J$ is fixed. These problems lead to different kinds of operator problems: linear PDEs in the first two cases and a nonlinear PDE in the latter case. These results are obtained through a variational principle using infinite dimensional Lagrange multipliers.
In the case where the initial dynamics are reversible the optimally controlled diffusion is also reversible. The particular case of constraining the dynamics to be reversible of the optimally controlled process leads to a linear eigenvalue problem for the square root of the density process.
\end{abstract}

\maketitle

\keywords{Key words and phrases: Stochastic optimal control, ergodic theory, calculus of variations, differential geometry, flux, current, gauge invariance}

\subjclass{AMS Subject Classification (2010): Primary 49K20; Secondary 93E20, 58A25}


\section{Introduction}

In this paper we discuss stochastic, long term average optimal, or `ergodic' control problems on compact orientable manifolds, in which control is exerted in all directions, and where the control costs are proportional to the square of the norm of the control field with respect to the metric induced by the noise. As such, our emphasis is not on the solution of applied control problems. However this setting has strong connections with the general theory of large deviations of ergodic Markov processes \cite{DonskerVaradhan1975, Varadhan1984}. We place special emphasis on the characterization of control solutions in terms of density and current, in relation to the so called Level 2.5 large deviations theory \cite{Chernyak2009, BaratoChetrite2015, Bertini2015}, discussed in more detail below.

The general theory of ergodic control in continuous spaces has been developed rigorously relatively recently; see works by Borkar and Gosh (e.g. \cite{BorkarGhosh1988}) and the recent monograph \cite{Arapostathis2012}. The special case of compact manifolds has been extensively studied in relation to the theory of large deviations. A brief historic overview of this field will be provided in this introduction (Section~\ref{sec:historic-overview}).

The `squared control cost' case is further motivated by recent attention to stochastic optimal control for finite time horizon problems with relative entropy determining control cost \cite{Kappen2005}. Typically the solution of stochastic optimal control problems can be rephrased as the solution of a non-linear partial differential equation called the Hamilton-Jacobi-Bellman equation \cite{FlemingRishel1975, FlemingSoner2009}. If the control cost can be interpreted as a relative entropy (in continuous settings, problems with squared control cost) this often yields elegant simplifications of the non-linear Hamilton-Jacobi-Bellman (HJB) equation through an exponential transform, known as the Cole-Hopf transform in fluid dynamics \cite{Hopf1950}. In the finite time horizon case the HJB equation is transformed into a linear equation, see e.g. \cite{Kappen2005, BierkensKappen2012b}. In the ergodic setting it leads typically to operator eigenvalue problems. 

The reader interested in the statistical physics interpretation of this material is referred to the related publication by the same authors~\cite{Chernyak2013} which provides a brief overview in physical terms of some of the main results in this paper, including the expression of ergodic behaviour in terms of current and density. The current paper can be seen as a more detailed and precise mathematical exposition of these results, and at some appropriate places we will point to related discussions in~\cite{Chernyak2013}. Perhaps the main achievement of this paper, in comparison with \cite{Chernyak2013}, is the precise mathematical characterization of the necessary and sufficiency conditions for optimality for the optimal control problem formulated over density and current (Section~\ref{sec:hjb}). Also we have a more extensive description of the related optimization problems in which one of the variables (current or density) is held fixed (Sections~\ref{sec:fixed_density},~\ref{sec:fixed_current_density}).

\subsection{Historical overview of connections between stochastic control theory and the theory of large deviations}
\label{sec:historic-overview}
Let us first remark that we will take some care to rephrase the historical results using the same notation as the one which is used in subsequent sections of this paper. As a basic introductions to the theory large deviations we recommend \cite{Varadhan1984, Hollander2000}.

In \cite{Kac1951} the classical expression 
\begin{equation} \label{eq:kac-expression} \lambda^{\star} = -\lim_{t \rightarrow \infty} \frac 1 t \log \E \left[ \exp \left( \int_0^t V(X_s) \ d s \right) \right]\end{equation}
(with expectations over Brownian sample paths)
was obtained for the principal eigenvalue $\lambda^{\star}$ of the eigenvalue problem
\begin{equation} \label{eq:kac-pde} \half \psi'' - V \psi = \lambda \psi,\end{equation}
where $V : \R \rightarrow [0,\infty)$.

As is well known, the representation~\eqref{eq:kac-expression} has a direct connection to the theory of large deviations. Specifically, for real random variables $Y_n$ a \emph{large deviation principle (LDP)} is said to hold with \emph{rate function} $I$ if 
\begin{equation}
 \P \left( Y_n \in A \right) \approx \exp\left( - n \inf_{y \in A} I(y) \right),
\end{equation}
in which case the \emph{rate function} $I$ satisfies, by Varadhan's Lemma (or in physics nomenclature, the Laplace method),
\[ \lim_{n \rightarrow \infty} \frac 1 n \log \E \left[ \exp \left( n F(Y_n) \right) \right] = \sup_{y \in \R} [F(y) - I(y)].\]
A converse result known as the G\"artner-Ellis theorem allows one, under certain conditions, to deduce an LDP from the Legendre transform of the large $n$ limit of the exponential expectations $\frac 1 n \log \E \left[ \exp \left( n F(Y_n) \right) \right]$. Bryc's formula \cite{Bryc1990} is an extension to infinite dimensions of this result.

In our case of interest, taking $Y_n = \frac 1 n \int_0^n V(X_s) \ d s$ and $F(y) = y$, the connection between~\eqref{eq:kac-expression} and the theory of large deviations becomes clear. At a higher level of abstraction, one could let $Y_n$ assume values in the space of probability measures, and let $Y_n$ denote the empirical distribution of $(X_s)$, i.e. $Y_n(A) = \frac 1 n \int_0^n \1_{\{X_s \in A\}} \ d s$. In this case the functional $F$ would be $F(\mu) := \int V \ d \mu$. Such a large deviation principle on the level of empirical distributions is called a \emph{Level 2 LDP}. At an even higher level of abstraction, the \emph{Level 3 LDP} concerns the asymptotic behaviour of the empirical process, i.e. the empirical distributions of finite sequences $(Y_1, \dots Y_n)$.

Recently there has been an emergence of interest in the so called \emph{Level 2.5 LDP}, which concerns the empirical distribution of a Markov sequence of random variables along with the empirical distribution of its current or flow. See e.g. \cite{Chernyak2009, Chernyak2013} for relations with non-equilibrium statistical mechanics, and \cite{BaratoChetrite2015, Bertini2015} for large deviations of currents in Markov processes. It is the aim of this paper to contribute to the Level 2.5 large deviation theory of diffusion processes in connection with a) the theory of stochastic control and b) certain partial differential equations and eigenvalue problems.

We hope that the above discussion has clarified to the reader the strong connection between an expression of the form~\eqref{eq:kac-expression} and the theory of large deviations, and will now sketch the historical developments concerning the principal eigenvalue in relation to stochastic control and large deviations below.

The seminal papers \cite{DonskerVaradhan1975-I, DonskerVaradhan1975} extended~\eqref{eq:kac-expression} to general Markov processes on compact metric spaces, in the following sense. Let $L$ be the infinitesemal generator of a Feller-Markov process in a compact metric space $E$ and let $V : E \rightarrow \R$ be a continuous mapping.
In \cite{DonskerVaradhan1975-I} a variational characterization of the principal eigenvalue $\lambda^{\star}$ of the operator $L - V$ is established as
\begin{equation} \label{eq:donsker-varadhan-1} \lambda^{\star} = - \inf_{\mu \in \mathcal M} \left[ \int_E V \ d \mu + I(\mu) \right],\end{equation} where $\mu$ is the space of probability measures on $E$ and $I$ is defined by
\begin{equation} \label{eq:donsker-varadhan-rate-function} I(\mu) = - \inf_{\substack{\psi \in D(L) \\ \psi > 0 }} \int_E \left( \frac{L \psi}{\psi} \right) \ d \mu.\end{equation}
The function $I$ is the rate function describing the large deviations of the empirical distribution from the invariant probability distribution.
In \cite{DonskerVaradhan1975} the relation to exponential expectations is obtained,
\begin{equation} \label{eq:donsker-varadhan-2} \lim_{t \rightarrow \infty} \frac 1 t \log \E \left[ \exp \left( - t F (L_t) \right) \right] = -\inf_{\mu \in \mathcal M} [ F(\mu) + I(\mu)],\end{equation}
with
\[ L_t(A) = \frac 1 t \int_0^t \1_{A} (X_s) \ d s,\] i.e. $L_t(A)$ denotes the proportion of time up to time $t$ that a sample path spends in $A$, and $F : \mathcal M \rightarrow \R$. Taking $F (\mu) := \int_E V \ d \mu$ formally recovers~\eqref{eq:kac-expression} in the special case of a Brownian motion.
An extension to non-compact spaces is given in \cite{DonskerVaradhan1976-III} but it seems that the stated conditions are hard to check in practice; see also \cite{Varadhan1984}.
In \cite{Gartner1977} an alternative characterization of the Donsker-Varadhan rate function for non-degenerate diffusions on compact manifolds was obtained independently, given by
\begin{equation} \label{eq:gartner} I(\mu) = \half \int_E \| \nabla \Phi \|^2 \ d \mu,\end{equation}
for $\mu$ having density $\rho$ with respect to Riemann volume, and where $\Phi$ is the unique (up to a constant) solution of 
\[ \Delta \Phi + \frac 1 \rho \langle \nabla \rho, \nabla \Phi \rangle = \frac 1 \rho L^{\star} \rho,\]
with $L^{\star}$ the adjoint of $L$.
In \cite{Holland1978} a similar representation for the principal eigenvalue of a diffusion generator is obtained, and in this paper the relation to stochastic control theory seems to be discussed for the first time. The connection with the control theoretic formulation is explored more extensively for the one-dimensional Brownian case in \cite{Karatzas1980}. It is noted that a change of variables $\psi = \exp(\Psi)$ in~\eqref{eq:kac-pde} yields
\[\lambda^{\star} = \half \frac{d^2 \Psi}{d x^2} + \min_u \left( u \frac{d \Psi}{d x} + \half u^2  \right) - V(x),\]
with the minimum being attained in $u^{\star} = -\frac{d \Psi}{d x}$.
The above equation can be recognized as the Hamilton-Jacobi-Bellman equation \cite{FlemingRishel1975, FlemingSoner2009} corresponding to a long term average cost problem with dynamics
\begin{equation} \label{eq:controlled-brownian-motion} d X_t^u = u_t(X^u) \ d t + d W_t, \quad X_0 = x, \end{equation}
where $W$ is a standard Brownian motion and where $u_t(X_t^u)$ is a control depending on $t \in [0,\infty)$ and the path of $X_t^u$ up to time $t$. The cost function of the associated control problem is 
\[ \mathcal C(u, x) = \lim_{T \rightarrow \infty} \frac 1 T \E_x \left[ \int_0^T V(X_t) + \half u_t(X)^2 \ d t \right].\]
Also in \cite{Karatzas1980} it is discussed how the principal eigenvalue $\lambda^{\star}$ admits the representation 
\[ \lambda^{\star} = -\lim_{T \rightarrow \infty} \frac{\Psi(x;T)}{T},\]
where $\Psi(\cdot;T)$ is the value of the optimal control problem with finite time horizon $T$, with dynamics~\eqref{eq:controlled-brownian-motion} and cost functional 
\[ \mathcal C(x,T; u) = \E_x^u \int_0^{T} \left\{ V(X_t^u) + \half U_t^2(X^u) \right\} \ d t.\]
Subsequently Fleming, Sheu and Soner clarified the relation between stochastic control and the principal eigenvalue for general Markov operators \cite{Fleming1982, Sheu1984}, and provided an alternative proof of Donsker and Varadhan's result on the relation of the principal eigenvalue to the exponential expectation~\eqref{eq:donsker-varadhan-2} in \cite{FlemingSheuSoner1987}; see also the book chapter \cite{Fleming1985}.

The connections between the theory of stochastic optimal control and large deviation theory extend beyond the ergodic setting which is considered in this paper. 
For large deviations theory of Markov chains in continuous time see \cite{Hollander2000, Fortelle2001}.
A modern research monograph on the connection between stochastic optimal control and large deviation theory is provided by \cite{Feng2006}. 
A different more commonly encountered stochastic control problem over an infinite time horizon is the control problem with discounted cost \cite[Section III.9]{FlemingSoner2009}. The ergodic control problem can be thought of as the limiting case where the discount factor approaches zero \cite{Arapostathis2012}. For applied control theoretic papers see also e.g. \cite{Rutquist2008} for the diffusion case and \cite{Todorov2006} for the Markov chain setting.

\subsection{Vector potentials and current density}

On a compact manifold, a few phenomena play a special role. The most important aspect of this setting is that transient behaviour cannot occur. Therefore, an invariant measure is necessarily unique and ergodicity follows immediately. The long term stochastic dynamics on the manifold will be completely characterized by the long term (particle) density $\rho$ and the long term current density $J$ (see Section~\ref{sec:ergodic_reformulation}). As such, control problems may be reformulated as variational problems over $\rho$ and $J$. We will consider a cost functional where the density $\rho$ is paired with the \emph{scalar cost} or \emph{(scalar) potential}  $V$, and the current density $J$ is paired with a \emph{vector potential} or \emph{gauge field} $A$. 

It has been well established long ago in quantum field theory that the vector potential $A$ plays the role of the variable conjugate to the current density $J$, just as the scalar potential $V$ is related to the charge density $\rho$. Conservation of current, also known as the continuity equation, can then be viewed as a dual formulation of gauge invariance. This duality is not specific to the quantum world and can be applied to currents in a more general setting, e.g., for stochastic currents \cite{Chernyak2009, BaratoChetrite2015}. In the case of stochastic processes adding the current density variables to the more customary particle density variable is often referred to as the $2.5$ level of theory \cite{BaratoChetrite2015}, and has shown several advantages. 

Besides the obvious benefit of an ability to efficiently treat observables depending on current, such as performed work, generated heat or entropy, as well as fluxes (see Section~\ref{sec:flux} and Appendix~\ref{app:flux}), the combination of scalar density and current density turns out to constitute the `right' set of variables in the sense that the large deviations of these quantities can be identified explicitly both in the continuous and discrete settings (such as stochastic dynamics on graphs). Even without the inclusion of a vector potential (i.e. taking $A = 0$), the formulation in terms of current and density provides a new and clear perspective on the optimal control problem in relation to the large deviations theory. The corresponding rate function is known in the physics literature as the current-density functional. Computation of the large deviations of the generalized observables can be reduced to solving an optimization problem, similar in spirit to the discussion in Section~\ref{sec:historic-overview} for the scalar potential above. 

In this manuscript we demonstrate that the gauge invariant approach, which has shown its capability in the context of stochastic processes (as briefly discussed above) can be extended to the stochastic optimal control setting while maintaining the same advantages of (i) considering optimal control of current density, in addition to classical particle density, which is achieved by introducing the gauge invariant extension of the standard Bellman equation, and (ii) formulating the gauge invariant Bellman equation as a solution of an optimization problem that involves a functional of current and particle densities. In other words, the approach outlined in this manuscript can be viewed as building the $2.5$ level stochastic optimal control theory.

\subsection{Outline and aims of this paper}

In Section~\ref{sec:problemsetting} the long term average stochastic optimal control problem for a non-degenerate diffusion over a compact manifold is formulated. In Section~\ref{sec:ergodic_reformulation} some preliminary operations, mostly based upon the ergodic properties of a compact diffusion, are performed which allow us to remove all reference to probability theory from the problem formulation, resulting in a simplified optimization problem.

We then obtain conditions related to optimality for the formulated ergodic optimization problem. In Section~\ref{sec:hjb} the problem in which both density $\rho$ and current density $J$ are varied freely is solved, in the sense that necessary as well as sufficient conditions for optimality are obtained. The sufficient conditions are obtained by considering the dual problem (Section \ref{sec:sufficient}) and the necessary conditions are derived using the method of Lagrange multipliers (Section \ref{sec:necessary}). These conditions can be phrased as the solution of a linear eigenvalue problem. As a side result we derive a variational formula for the principal eigenvalue of an elliptic operator.

Then we consider two further, closely related optimization problems. In Section~\ref{sec:fixed_density} the optimization problem in which $J$ is varied for fixed density $\rho$ is discussed. It turns out that a necessary condition for optimality can be phrased as a linear elliptic PDE. In Section~\ref{sec:fixed_current_density} we obtain necessary conditions for optimality of the optimization problem for fixed current density $J$ and varying $\rho$. Here the necessary condition for optimality can be formulated as a non-linear eigenvalue problem. In the reversible case this reduces to a linear eigenvalue problem. 

In the case where the initial dynamics are reversible we obtain the result that the optimally controlled diffusion is also reversible (Section~\ref{sec:reversible_solution}). The particular case of insisting $J = 0$ coincides with demanding reversible dynamics of the optimally controlled process. Interestingly, this optimization problem leads to a linear eigenvalue problem for the square root of the density process, just as we see in quantum mechanics (but note that our setting is entirely classical). We conclude this paper with a brief discussion (Section~\ref{sec:discussion}). In the appendices a detailed discussion of the use of a vector potential to quantify flux is provided (Appendix~\ref{app:flux}), as well as a derivation of the relation between long term average drift, density, and current density (Appendix~\ref{app:long_term_current}).


\subsection{Notation}
When $(\Omega, \mathcal F, \P)$ is a probability space and $H(\omega, x_1, x_2, \dots)$ is a measurable function of $\omega \in \Omega$ and other variables $x_1, x_2, \dots$, we often omit the dependence on $\omega$, making $H(x_1, x_2, \dots)$ into a random quantity which also depends on $x_1, x_2, \dots$.

Throughout the paper we will work on a smooth orientable Riemannian manifold $(M,g)$ and use similar notation as may be found e.g. in \cite[Chapter V]{IkedaWatanabe1989} or \cite{Warner1983}. By smooth we mean infinitely often differentiable, unless stated otherwise. As usual we employ the Einstein summation convention, i.e. in local coordinates summation over one upper and one lower index is automatic, e.g. $g^{ij} f_j = \sum_j g^{ij} f_j$.
$C^{\infty}(M)$ denotes the space of smooth functions from $M$ into $\R$, $C^k(M)$ the space of $k$ times continuously differentiable functions (where $k = 0, 1, \dots$), $\mathfrak X(M)$ denotes the space of smooth vectorfields on $M$, and $\Lambda^p(M)$ denotes the space of smooth differential forms of order $p$ on $M$, for $p = 0, 1, \hdots, n$. The volume form is denoted $d x$, in local coordinates $d x = \sqrt{\det g_{ij}(x)} d x^1 \dots d x^n$. The Riemannian metric induces a local inner product $\langle \cdot, \cdot \rangle$ and corresponding norm $\|\cdot \|$ on tensors of arbitrary covariant and contravariant orders. E.g.  if $S^{i_1, \dots, i_p}_{j_1, \dots, j_q}$ and $T^{i_1, \dots, i_p}_{j_1, \dots j_q}$ are $(p,q)$-tensors, we have 
\[ \langle S, T \rangle = g_{i_1 i_1'} \dots g_{i_p i'_p} g^{j_1 j'_1} \dots g^{j_q j'_q} S^{i_1, \dots, i_p}_{j_1, \dots, j_q} T^{i'_1, \dots, i'_p}_{j'_1, \dots j'_q}\]
and $\| T \| = \sqrt{\langle T, T\rangle}$.
We employ the usual notions of exterior derivative $d : \Lambda^{p-1}(M) \rightarrow \Lambda^p(M)$, which as  $L^2(M,g)$ adjoint $\delta : \Lambda^p(M) \rightarrow \Lambda^{p-1}(M)$, i.e. for  differential forms $\alpha \in \Lambda^p$ and $\beta \in \Lambda^{p-1}$, we have
\[ \int_M \langle \alpha, d \beta \rangle \ d x = \int_M \langle \delta \alpha, \beta \rangle \ d x.\]
Partial derivatives in local coordinates are denoted by $\partial_i = \frac{\partial}{\partial x^i}$. If $\xi \in \mathfrak X(M)$ and $f \in C^{\infty}(M)$ then in local coordinates $(\xi f) (x) = \xi^i \partial_i f(x)$.
The symbol $\nabla$ is the covariant derivative corresponding to the Levi-Civita connection of the Riemannian metric $g$. The Christoffel symbols corresponding to the Levi-Civita connection are denoted by $\Gamma^{i}_{jk}$.
The gradient of a function $f \in C^{\infty}(M)$ is the vector field $\grad f = \nabla f$ with components $g^{ij} \partial_j f$. The divergence of a vectorfield $\xi \in \mathfrak X(M)$ is, in local coordinates, $\div \xi = \frac 1 { \sqrt{ \det G}} \partial_i(\xi^i \sqrt{\det G})$, where $G = (g_{kl})$, and satisfies $\div \xi = - \delta \alpha$, where $\alpha \in \Lambda^1(M)$ is given by $\alpha_i = g_{ij} \xi^j$.
On $\Lambda^p(M)$ an inner product is defined by $\langle \alpha, \beta \rangle_{\Lambda^p(M)} = \int_M \langle \alpha, \beta \rangle \ d x$, where $d x$ denotes the volume form corresponding to $g$. The inner product $\langle \cdot,\cdot\rangle_{\Lambda^0(M)}$ is also denoted by $\langle \cdot, \cdot \rangle_{L^2(M)}$. Let $L^2(M, g) = L^2(M)$ denote the usual Hilbert space obtained by completing $C^{\infty}(M)$ with respect to the $L^2(M)$-inner product and considering, where appropriate, equivalence classes of functions. The Hodge star operator is denoted by $\star : \Lambda^p(M) \rightarrow \Lambda^{n-p}(M)$, $p = 0, \hdots, n$.

\section{Problem setting}
\label{sec:problemsetting}
Throughout this paper let $(M, g)$ denote a smooth compact connected oriented $n$-dimensional Riemannian manifold. 
Let $(\Omega, \mathcal F, (\mathcal F_t), \P)$ denote a filtered probability space on which is defined a $d$-dimensional standard Brownian motion. Consider a stochastic process $X$ defined on $M$ by the SDE, given in local coordinates by
\begin{equation} \label{eq:uncontrolled} 
d X_t = b(X_t) + \sum_{\alpha = 1}^d \sigma_{\alpha}  \circ d B^{\alpha}_t, \quad t \geq 0,
\end{equation}
or, in local coordinates,
\[ d X^i_t =  b^i(X_t) \ d t + \sum_{\alpha = 1}^d \sigma^i_{\alpha}(X_t) \circ d B^{\alpha}_t, \quad t \geq 0,\]
where, for $\alpha = 1, \hdots, d$, $\sigma_{\alpha} \in \mathfrak X(M)$, and it is assumed without loss of generality that the noise vectorfields $\sigma_{\alpha}$ are related to the Riemannian metric through the relation 
\[ g^{ij} = \sum_{\alpha=1}^d \sigma^i_{\alpha} \sigma^j_{\alpha}, \quad i, j = 1, \hdots, n.\]
The notation $\circ \, d B^{\alpha}$ indicates that we take Stratonovich integrals with respect to the Brownian motion. 
One can think of $b$ as a force field, resulting from a potential, some external influence, or a combination of both.


The SDE~\eqref{eq:uncontrolled} is referred to as the \emph{uncontrolled dynamics}. 
These dynamics may be altered by exterting `control' vectorfield $u \in \mathfrak X(M)$ in the following way, 
\begin{equation} \label{eq:controlled} d X_t = \left[ b(X_t) + u(X_t) \right] \ d t + \sum_{\alpha = 1}^d \sigma_{\alpha}(X_t) \circ d B^{\alpha}_t,  \quad t \geq 0. \end{equation}
For any initial condition $X_0 = x_0 \in M$ and control vectorfield $u \in \mathfrak X(M)$, a unique solution to~\eqref{eq:controlled} exists \cite[Chapter V]{IkedaWatanabe1989} and will be denoted by $X^{x_0, u}$.
The SDE~\eqref{eq:controlled} is referred to as the \emph{controlled dynamics}. 

Consider the random functional $\mathcal C : \Omega \times M \times \mathfrak X(M) \rightarrow \R$ denoting pathwise long term average cost,
\begin{equation} \label{eq:costfunction}
 \mathcal C(\omega, x_0, u) := \limsup_{T \rightarrow \infty} \frac 1 T \left[ \int_0^T V(X_s^{x_0,u}) + \frac 1 2 \| u(X_s^{x_0,u})\|^2 \ d s + \int_0^T \langle A(X_s^{x_0,u}),  \circ d X_s^{x_0,u} \rangle \right]
\end{equation}
where $V \in C^{\infty}(M)$ is a \emph{potential} or \emph{state dependent cost function}, $\|u(\cdot)\|^2$ represents the \emph{(instantaneous) control cost} corresponding to a control vectorfield $u \in \mathfrak X(M)$, and $A \in \mathfrak X(M)$. The final term in~\eqref{eq:controlled} is of course shorthand notation for
$\int_0^T [g_{ij} A^i](X_s^{x_0,u})  \circ d (X_s^{x_{0},u})^j$ and may represent a \emph{flux}, as explained in Section~\ref{sec:flux} and more extensively in Appendix~\ref{app:flux}. The vectorfield $A$ is often called a \emph{vector potential} or \emph{gauge field} in physics.

\begin{remark}
\label{rem:A_actually_diff_form}
From a physics perspective, it would be better to let the vector potential $A$ take the form of a differential form. On a mathematical level this distinction is irrelevant and we choose $A$ to be a vectorfield for notational convenience, unless stated otherwise.
\end{remark}

\begin{remark}
The notation `$\limsup$' in~\eqref{eq:costfunction} is used to avoid any discussion at this point about the existence of the limit. Instead of the \emph{pathwise formulation} in~\eqref{eq:costfunction}, we could alternatively consider the weaker \emph{average formulation}, in which case the cost function would be the long term average of the expectation value $\E^{x_0,u}$ of the integrand in~\eqref{eq:costfunction}. We will see in Section~\ref{sec:ergodic_reformulation} that the limit of~\eqref{eq:costfunction} exists (and not just the `$\limsup$'). Furthermore this limit will turn out to be equal to a deterministic quantity (after excluding a a set of measure zero with respect to the invariant distribution), so that the pathwise formulation and the average formulation may be considered equivalent.
\end{remark}

It is the main aim of this paper to consider the following problem. 

\begin{problem}
\label{prob:original_formulation}
For every $x_0 \in M$, find a control vectorfield $\hat u \in \mathfrak X(M)$ such that
\[ \mathcal C(x_0, \hat u) = \inf_{u \in \mathfrak X(M)} \mathcal C(x_0, u), \quad \mbox{almost surely.}\]
\end{problem}

In Sections~\ref{sec:fixed_density} and \ref{sec:fixed_current_density} we will also discuss other variants of the control problem, where we will respectively fix the invariant density and the current density, which will be defined in Section~\ref{sec:ergodic_reformulation}.

\section{Ergodic reformulation of the optimization problem}
\label{sec:ergodic_reformulation}
In this section we will derive two equivalent formulations of Problem~\ref{prob:original_formulation}. These reformulations, Problem~\ref{prob:ergodic_minimization} and Problem~\ref{prob:ergodic_minimization_with_current} below, are better suited to the analysis in the remaining sections. Also some notation will be established that will be used throughout this paper.

Let $\Omega^X = C([0,\infty);M)$ denote the space of sample paths of solutions to~\eqref{eq:controlled}. 
We equip $\Omega^X$ with the $\sigma$-algebra $\mathcal F^X$ and filtration $(\mathcal F_t^X)_{t \geq 0}$ generated by the cylinder sets of $X$. Furthermore let probability measures $\P^{x_0,u}$ on $\Omega^X$ be defined as the law of $X^{x_0, u}$, for all $x_0 \in M$ and $u \in \mathfrak X(M)$. For every  $u \in \mathfrak X(M)$ the collection of probability measures $\P^{\cdot, u}$ defines a Markov process on $\Omega^X$, i.e. for every $x_0, x_1 \in M$, 
\[ \P^{x_0,u} \left((X(t_1 + s), \dots, X(t_k + s)) \in F) \mid \mathcal F_s^X \right) = \P^{X(s),u} \left(( X(t_1), \dots, X(t_k)) \in F \right).\]
For the moment let $u \in \mathfrak X(M)$ be fixed.
By~\cite[Theorem V.1.2]{IkedaWatanabe1989}, the Markov generator corresponding to~\eqref{eq:controlled} is given by
\[ L_u f(x) = \half \sum_{\alpha = 1}^d \sigma_{\alpha} \sigma_{\alpha} f(x) + (b + u) f(x).\]


\begin{lemma}
$L_u$ may be written as
\[ L_u f = \half \Delta f + (\widetilde b + u) f, \quad f \in C^2(M),\]
where $\widetilde b := b +  \sum_{\alpha = 1}^d \nabla_{\sigma_{\alpha}} \sigma_{\alpha}$.
The adjoint of $L_u$ with respect to the $L^2(M)$ inner product is given by
\[ L_u^{\star} \rho = \half \Delta \rho - \div \left( \rho (\widetilde b + u)  \right), \quad \rho \in C^2(M).\]
 \end{lemma}

\begin{proof}
The Laplace-Beltrami operator may be expressed as (see \cite[p. 285, eqn. (4.32)]{IkedaWatanabe1989})
\begin{equation}
 \label{eq:laplace_beltrami} \Delta f = g^{ij} \partial_i \partial_j f - g^{ij} \Gamma^{k}_{ij} \partial_k f.
\end{equation}
Using this expression, we compute
\begin{align*} \sum_{\alpha = 1}^d \sigma_{\alpha} \sigma_{\alpha} f & = \sum_{\alpha=1}^d \sigma_{\alpha}^i \partial_i \left( \sigma_{\alpha}^j \partial_j f \right) = \sum_{\alpha=1}^d \left( \sigma_{\alpha}^i \sigma_{\alpha}^j \partial_i \partial_j f + \sigma_{\alpha}^i \left( \partial_i \sigma_{\alpha}^j \right) \left( \partial_j  f\right) \right) \\
& = \sum_{\alpha=1}^d\left( g^{ij} \partial_i \partial_j f + \sigma_{\alpha}^i \left( \partial_i \sigma_{\alpha}^j \right) \left( \partial_j  f\right) \right)  = \Delta f + g^{ij} \Gamma^k_{ij} \partial_k f + \sum_{\alpha=1}^d \sigma_{\alpha}^i \left( \partial_i \sigma_{\alpha}^j \right) \partial_j f \\ &  = \Delta f + \sum_{\alpha=1}^d \left(\nabla_{\sigma_{\alpha}} \sigma_{\alpha}\right) f,
\end{align*}
where the last equality is a result of the definition of the Levi-Civita connection and the corresponding Christoffel symbols. The expression for $L_u^{\star}$ is immediate from its definition.
\end{proof}

In the remainder of this work, we will assume that all advection terms are absorbed in the drift $b$ so that we may omit the tilde in $\widetilde b$. This can alternatively be interpreted as assuming $\sum_{\alpha=1}^d \nabla_{\sigma_{\alpha}} \sigma_{\alpha} = 0$. This is further equivalent to demanding that $L_u = \half \Delta + b + u$. This assumption is without loss of generality (on the level of the probability law on trajectories of $X$) by the above lemma and the fact that the Markov generator $L^u$ uniquely determines the law of the trajectories of $X$.
%

\begin{lemma}
Let $x_0 \in M$ and $u \in \mathfrak X(M)$. The expectation of the trajectory of $X$ over the vector potential may be expressed as
\label{lem:expectation_gaugefield}
\[ \E^{x_0,u} \int_0^T \langle A(X_t), \circ d X_t  \rangle = \E^{x_0,u} \int_0^T \left[ \langle A, (b + u) \rangle + \half \div A \right](X_t) \ d t.\]
\end{lemma}

\begin{proof}
Using the usual transformation rule between It\^o and Stratonovich integrals \cite[Equation (1.4), p. 250]{IkedaWatanabe1989}, we may write
\begin{equation} \label{eq:ito_equation} d X_t^i = \overline b_u^i(X_t) \ d t + \sum_{\alpha = 1}^d \sigma^i_{\alpha}(X_t) d B^{\alpha}(t), \end{equation}
where $\overline b_u(x)$ is given by
\[ \overline b_u^i(x) := b^i(x) + u^i(x) + \half \sum_{\alpha = 1}^d \left( \partial_k \sigma^i_{\alpha}(x)  \right)\sigma^k_{\alpha}(x).\]
By the definition of the Stratonovich integral, $Z \circ d Y = Z \ d Y + \half d [Z, Y]$ for semimartingales $Y$ and $Z$ \cite[Equation (1.10), p.100]{IkedaWatanabe1989}, with $Z \ d Y$ denoting the It\^o integral. Therefore
\begin{align*} & [g_{ij} A^i](X_t) \circ d X_t^j \\
& = [g_{ij} A^i](X_t) \ d X_t^j + \half d [(g_{ij} A^i)(X_t), X_t^j] \\
& = [g_{ij} A^i(X_t)] \ d X^j_t + \half \partial_k (g_{ij} A^i)(X_t) d [X_t^k, X_t^j] \\
& = [g_{ij} A^i(X_t)] \left[ \overline b_u^j(X_t) \ d t + \sum_{\alpha=1}^d \sigma_{\alpha}^i(X_t) \ d B^{\alpha}(t) \right] + \half \partial_k (g_{ij} A^i)(X_t)  \sum_{\alpha = 1}^d  \sigma_{\alpha}^k(X_t) \sigma_{\alpha}^j(X_t) \ d t.
\end{align*}
Integrating over $t$ and taking expectations gives 
\begin{align*} & \E^{x_0,u} \int_0^T A(X_s) \circ d X_s \\
&  = \E^{x_0, u} \int_0^T \left\{ (g_{ij} A^i)(X_t) \left[ b_u^j(X_t) + \half \left( \partial_k\sigma^j_{\alpha}(X_t) \right) \sigma^k_{\alpha}(X_t) \right] +\half \partial_k (g_{ij} A^i)(X_t)  \sum_{\alpha = 1}^d  \sigma_{\alpha}^k(X_t) \sigma_{\alpha}^j(X_t)  \right\} \ d t \\
& = \E^{x_0,u}\int_0^T  \left\{ g_{ij} A^i (b^j + u^j)(X_t) + \half \nabla_i  A^i (X_t)  \right\} \ d t.
\end{align*}
In the last expression we recognize the divergence of the vectorfield $A^i$, resulting in the stated expression.
\end{proof}
Let $\mathcal B(M)$ denote the Borel $\sigma$-algebra on $M$. 
Let $u \in \mathfrak X(M)$.
A probability measure $\mu_u \ d x$ on $M$ is called an \emph{invariant probability distribution} for~\eqref{eq:controlled}, if
\[ \int_M \P^{x, u}(X_t \in B) \mu_u \ d x(dx) = \mu_u(B), \quad \mbox{for all $t \geq 0$ and $B \in \mathcal B(M)$}.\]
The following result on invariant measures for non-degenerate diffusions \cite[Proposition V.4.5]{IkedaWatanabe1989} is essential for our purposes.
\begin{proposition}[Existence and uniqueness of invariant probability measure]
\label{prop:existence_uniqueness_inv_measure}
For any $u \in \mathfrak X(M)$ there exists a unique invariant probability measure $\rho_u \ d x$ on $M$ for~\eqref{eq:controlled} which is absolutely continuous with respect to the Riemannian volume measure $d x$.
The density $\rho_u \in C^{\infty}(M)$ is a solution of
\begin{equation} \label{eq:fokker_planck} L^{\star}_u \rho = 0.\end{equation}
Furthermore $\rho_u > 0$ on $M$.
\end{proposition}

We will refer to~\eqref{eq:fokker_planck} as the \emph{Fokker-Planck equation}, in agreement with common physics terminology. In the remainder of this work let $\rho_u$ as defined by Proposition~\ref{prop:existence_uniqueness_inv_measure}.


In the physics literature, the \emph{empirical density} and \emph{empirical current density} are defined respectively as (see \cite{Chernyak2009}):
\[ \rho_t(x, \omega) = \frac 1 t \int_0^t \delta(x - X_s(\omega)) \ d s, \quad J_t(x, \omega) = \frac 1 t \int_0^t \dot X_s \delta(x - X_s(\omega))  \ d s.\] Here (and only here) $\delta$ denotes the Dirac delta function.
These fields, which have a clear intuitive meaning, will be very relevant in the remainder of this work and we will make these precise from a mathematical point of view.

Let $B_b(M)$ denote the set of bounded Borel-measurable functions on $M$. 
We will work with the set of empirical average measures $\left(\nu_t(dx, \omega)\right)_{t > 0}$ on $\mathcal B(M) \times \Omega^X$, defined by
\begin{equation}
\label{eq:empirical_average_measure}
\nu_{t}(B) := \frac 1 t \int_0^t \1_B(X_s) \ d s, \quad t > 0, \ B \in \mathcal B(M),
\end{equation}
where $\1_{B}$ denotes the indicator function of the set $B$. 
Our primary interest is in the infinite time horizon limit. 

\begin{proposition}
\label{prop:long_term_density}
Let $u \in \mathfrak X(M)$. For all $\varphi \in L^2(M, \rho_u \ d x)$ and $\rho_u$-almost all $x_0 \in M$,
\begin{equation}
\label{eq:ergodic_theorem} \lim_{t \rightarrow \infty} \int_M \varphi \ d \nu_t = \int_M \varphi \ \rho_u \ d x, \quad \mbox{$\P^{x_0,u}$-almost surely.}
\end{equation}
\end{proposition}

\begin{proof}
For $u \in \mathfrak X(M)$, we define a stationary probability measure $\P^u$ on $\Omega^X$ by
\[ \P^u(G) = \int_M \P^{x,u}(G) \ \rho_u(x) \ dx, \quad G \in \mathcal F^X.\]
For $\varphi \in L^2(M, \rho_u \ d x)$ we then have, by the ergodic theorem, see e.g. \cite[Theorem 3.3.1]{dapratozabczyk1996}, that $\lim_{t \rightarrow \infty} \int_M \varphi \ d \nu_t = \int_M \varphi \  \rho_u \ d x,$ $\P^u$-almost surely.
Since $\rho_u > 0$ on $M$, this implies that 
\[ \lim_{t \rightarrow \infty} \int_M \varphi \ d \nu_t = \int_M \varphi \ \rho_u \ d x, \quad \mbox{$\P^{x_0,u}$-a.s. $\rho_u$-a.a. $x_0 \in M$.}\]
\end{proof}

Hence $\lim_{t \rightarrow \infty} \nu_t$ is $\mu_u$-almost everywhere equal to a constant.

\begin{corollary} 
\label{cor:ergodic_gaugefield}
Let $u \in \mathfrak X(M)$. For $\rho_u$-almost every $x_0 \in M$,
\begin{equation} \label{eq:ergodic_gaugefield} \lim_{T \rightarrow \infty} \frac 1 T \int_0^T \langle A(X_t), \circ d X_t \rangle = \int_M \left[  \langle A, (b+u) \rangle  + \half \div A \right] \rho_u \ d x,  \quad \mbox{$\P^{x_0,u}$-a.s.}\end{equation}
\end{corollary}
\begin{proof}
This is an immediate corollary of Lemma~\ref{lem:expectation_gaugefield}, Proposition~\ref{prop:existence_uniqueness_inv_measure} and Proposition~\ref{prop:long_term_density}.
\end{proof}

The above results provide sufficient motivation to rephrase Problem~\ref{prob:original_formulation} as follows.

\begin{problem}
\label{prob:ergodic_minimization}
Minimize 
\begin{equation} \label{eq:costfunction_ergodic} 
 \mathcal C(\rho, u) := \int_M \left\{ V + \half \|u\|^2 +  \langle A, (b+u) \rangle  + \half \div A \right\} \rho \ d x.
\end{equation}
 with respect to $(\rho, u) \in C^{\infty}(M) \times \mathfrak X(M)$ subject to the constraints $L_u^{\star} \rho = 0$ and $\int_M \rho \ d x = 1$.
\end{problem}

\begin{lemma}
\label{lem:remove_A}
Suppose $u, \widetilde u, b, \widetilde b \in \mathfrak X(M)$ and $V, \widetilde V \in C^{\infty}(M)$ are related by 
\begin{equation} \label{eq:b_and_V_no_A}  \widetilde b = b - A, \quad \widetilde u = u + A, \quad \mbox{and} \quad \widetilde V = V - \half \|A\|^2 + \langle A, b \rangle + \half \div A. \end{equation}
Let $\widetilde L_{\widetilde u}$ denote the generator corresponding to drift $\widetilde b$ and control $\widetilde u$. 
Then $L_u^{\star} \rho = 0$, $\int_M \rho \ d x = 1$, and $(\rho, u)$ are optimal for Problem~\ref{prob:ergodic_minimization}, if and only if 
$\widetilde L_{\widetilde u}^{\star} \rho = 0$ and $(\rho,\widetilde u)$ are optimal for Problem~\ref{prob:ergodic_minimization}, with $\mathcal C$ replaced by 
\[ \widetilde {\mathcal C}(\rho, \widetilde u) = \int_M \left\{ \widetilde V + \half \|\widetilde u \|^2 \right\} \rho \ d x.\]
\end{lemma}

\begin{proof}
First note that $\widetilde b + \widetilde u = b + u$, so that $L^{\star}_u = \widetilde L^{\star}_{\widetilde u}$, and hence $\widetilde L^{\star}_{\widetilde u} \rho = 0$ if and only if $L^{\star}_u \rho = 0$.
Furthermore $\widetilde C(\rho, \widetilde u) = C(\rho,u)$ by direct computation.
\end{proof}

\begin{remark}
\label{rem:remove_A}
As a consequence of Lemma~\ref{lem:remove_A} the vector potential $A$ may be completely removed from the problem by redefining $b$, $V$ and $u$ by~\eqref{eq:b_and_V_no_A}.
Using this observation simplifies the derivation of some of the results in subsequent sections.
\end{remark}

Note that $\dot X_s$ is not defined, a.s., so our mathematical analogue of the empirical current density requires more care.
In Appendix~\ref{app:long_term_current}, we derive  the vector field $J \in \mathfrak X(M)$ denoting current density, as 
\begin{equation}
\label{eq:long_term_current_density}
J = -\half \nabla \rho + \rho \left(b + u \right),
\end{equation}
for $\rho = \rho_u$.

By rearranging~\eqref{eq:long_term_current_density}, we can express a control $u$ in terms of $J$ and $\rho$ as
\begin{equation} \label{eq:u_intermsof_J} u = -b + \frac 1 {\rho} \left( J + \half \nabla \rho \right).\end{equation}
The following lemma is an immediate consequence of the equality $\div J = -L_u^{\star} \rho$.

\begin{lemma}
\label{lem:fokker_planck_equivalence}
Suppose $u$, $\rho$ and $J$ are related by~\eqref{eq:long_term_current_density}. Then $\div J = 0$ if and only if $L_u^{\star} \rho = 0$.
\end{lemma}
In other words, $\div J = 0$ if and only if $\rho \ dx$ is invariant for~\eqref{eq:controlled}. Similar observations may be found throughout the physics literature, e.g. \cite{Risken1989, BaratoChetrite2015}.
Recall Lemma~\ref{lem:expectation_gaugefield}, where the expectation of the vector potential $A$ over the trajectory was expressed as an expectation over a Lebesgue integral. For the long term average of the gauge field this leads to the following result.

\begin{lemma}
\label{lem:ergodic_gaugefield_current}
Suppose $u \in \mathfrak X(M)$ and $x_0 \in M$, and $J$ satisfy~\eqref{eq:long_term_current_density}. Then 
\begin{equation} \label{eq:longterm_gaugefield_current} \lim_{T \rightarrow \infty} \frac 1 T \int_0^T \langle A(X_t), \circ d X_t \rangle = \int_M \langle A, J \rangle \ d x, \quad \mbox{$\P^{x_0,u}$-a.s.}\end{equation}
\end{lemma}

\begin{proof}
From Corollary~\ref{cor:ergodic_gaugefield} we have~\eqref{eq:ergodic_gaugefield}.
By~\eqref{eq:u_intermsof_J}, and partial integration, this equals 
\begin{align*}
\int_M \left( \left \langle A, \frac 1{\rho} \left( J + \half \nabla \rho \right) \right \rangle + \half \div A \right) \rho \ d x & = \int_M \langle A, J \rangle \ d x.
\end{align*}
\end{proof}

Because of the above observations, instead of varying $\rho$ and $u$ in the optimization problem~\ref{prob:ergodic_minimization}, we may as well vary $\rho \in C^{\infty}$ and $J \in \mathfrak X(M)$, while enforcing $\div J = 0$ (equivalent to the Fokker-Planck equation for $\rho$ by Lemma~\ref{lem:fokker_planck_equivalence}) and $\int_M \rho \ d x = 1$.  The uniqueness of the solution to the Fokker-Planck equation ensures that $\rho$ is positive.
The control $u$ is then determined uniquely by \eqref{eq:u_intermsof_J}.
Combining~\eqref{eq:u_intermsof_J} and ~\eqref{eq:longterm_gaugefield_current}, we may alternatively express the cost functional~\eqref{eq:costfunction_ergodic} as a function of $\rho$ and $J$, namely
\begin{equation}
\label{eq:costfunction_intermsof_J_and_rho}
\mathcal C(\rho, J) = \int_M \left[ \left( V + \frac 1 {2} \left\| \frac 1 {\rho} \left(J +\half \nabla  \rho \right) - b \right\|^2 \right) \rho + \langle A, J \rangle \right] \ d x.
\end{equation}

\begin{remark} Strictly speaking the use of $\mathcal C$ for different cost functionals is an abuse of notation but we trust this will not lead to confusion.
\end{remark}

Problem~\eqref{prob:original_formulation} can thus be rephrased as the following problem:

\begin{problem}
\label{prob:ergodic_minimization_with_current}
Minimize $\mathcal C(\rho, J)$ with respect to $\rho \in C^{\infty}(M)$ and $J \in \mathfrak X(M)$, subject to the constraints $\div J = 0$ and $\int_M \rho \ d x = 1$, where $\mathcal C(\rho, J)$ is given by~\eqref{eq:costfunction_intermsof_J_and_rho}.
\end{problem}

Although we will not directly make use of the following fact, it seems relevant enough to mention here. 
\begin{proposition}
 $\mathcal C(\rho,J)$ is convex.
\end{proposition}
\begin{proof}
First consider the integrand pointwise, in the form of the map $h : \R \times \R^n \times \R^n \rightarrow \R$ given by
\[ h(x, y, z) = \left(a + \half \left\| c + \frac {y + \half z} x \right\|^2 \right) x + \langle m, y \rangle= a x + \half \| c\|^2 x + \langle c, y + \half z \rangle + \frac{\| y + \half z\|^2}{2 x} + \langle m, y \rangle.\]
This is a summation of convex functions in $(x,y,z)$, so it follows that $h$ is convex. Note that $\mathcal C$ is obtained by integrating $h$ over $M$, taking $x = \rho$, $y = J$ and $z = \nabla \rho$.
\end{proof}

\subsection{Interpretation of the vector potential as flux}
\label{sec:flux}
In view of Remark~\ref{rem:A_actually_diff_form}, for the discussion in this section let $A \in \Lambda^1(M)$ be a vector potential in differential form.
A particular use or interpretation of $A$ is that $A(J)$ may quantify flux of $J$ through a submanifold of $M$. In particular, for a given $(n - 1)$-cycle $\alpha$ (roughly speaking, an $(n-1)$-dimensional submanifold of $M$ without boundary), there exists a unique harmonic $A \in \Lambda^1(M)$ (i.e. $\Delta A = 0$, where $\Delta$ denotes the Laplace-Beltrami operator), which depends only on the singular homology class $[\alpha] \in \mathcal H_{m-1}(M;\R)$ of $\alpha$, such that $\int_{\alpha} \star J = \int_M A(J)  \ d x$ for all $J \in \mathfrak X(M)$ satisfying $\div J = 0$.

\begin{example}[$S^1$] The divergence free $1$-forms $J$ on $S^1$ are constant, say $J = J_0 \ d \theta$ for some $J_0 \in \R$. A 0-cycle $\alpha$ of $S^1$ consists of a collection of points $\theta_1, \hdots, \theta_k \subset [0, 2 \pi)$ with multiplicities $\alpha_1, \hdots, \alpha_k$. The flux of $J$ through $\alpha$ is then simply given by $\sum_{i=1}^k \alpha_i J(\theta_i) = \sum_{i=1}^k \alpha_i J_0$. By defining a differential form $A = A_0 \ d \theta$, with constant component $A_0 := \sum_i \alpha_i \frac 1 {2 \pi}$, we find that
\[ \int_{S^1} \langle A, J \rangle \ d \theta = \int_{S^1} A_0 J_0 \ d \theta = \sum_i \alpha_i J_0 = \int_{\alpha} \star J.\]
We see that this choice of $A$ is the constant (and therefore harmonic) representative in $\mathcal H^1_{\deRham}(S^1)$ corresponding to $[\alpha] \in \mathcal H_0(S^1;\R)$.
\end{example}

A more extensive discussion of this topic may be found in Appendix~\ref{app:flux}.

\section{Unconstrained optimization -- the HJB equation}
\label{sec:hjb}

In this section we will find necessary conditions for a solution of Problem~\ref{prob:ergodic_minimization} or equivalently Problem~\ref{prob:ergodic_minimization_with_current}. In fact, for technical reasons we will work with the the formulation in terms of $\rho$ and $J$, i.e. Problem~\ref{prob:ergodic_minimization_with_current}. The main reason for this is the convenient form (in particular, the linearity) of the constraint $\div J = 0$. This may be compared to the equivalent constraint $\half \Delta \rho - \div  (\rho (b + u)) = 0$, which is non-linear as a function in $(\rho, u)$.

The approach to Problem~\ref{prob:ergodic_minimization} or Problem~\ref{prob:ergodic_minimization_with_current} is to use the method of Lagrange multipliers to enforce the constraints. Since the constraint $\div J(x) = 0$ needs to be enforced for all $x \in M$, the corresponding Lagrange multiplier is an element of a function space. A purely formal derivation of the necessary conditions using Lagrange multipliers is straightforward, but we wanted to be precise in proving necessary and sufficient conditions for optimality. 

\subsection{Sufficient condition for optimality}
\label{sec:sufficient}
Define $\mathcal U = \{ (\rho, J) \in C^{\infty} \times \mathfrak X(M) : \rho(x) \geq  0 \ \mbox{for all $x \in M$}\}$. 
Associated to Problem~\ref{prob:ergodic_minimization_with_current} we may define a Lagrangian 
\begin{equation}
 \label{eq:lagrangian}
 \mathcal L(\rho, J, \Psi, \lambda) := \mathcal C(\rho, J) + \int_M \Psi (\div J) \ d x + \lambda \left(  \int_M \rho \,  d x - 1 \right),
\end{equation}
defined for $(\rho, J) \in \mathcal U$, $\Psi \in C^{\infty}(M)$ and $\lambda \in \R$. The dual cost functional is as usual defined as
\[ \mathcal C^{\star}(\Psi, \lambda) := \inf_{(\rho, J) \in \mathcal U} \mathcal L(\rho, J, \Psi, \lambda).\]
Write $\mathcal V := C^{\infty}(M) \times \R$ for the domain of the dual cost functional $\mathcal C^{\star}$. It is immediate from the definition that for every $(\rho, J) \in \mathcal U$ satisfying the constraints 
\begin{equation}
 \label{eq:constraints} \div J = 0 \quad \mbox{and} \quad \int_M \rho \ d x = 1,
\end{equation}
 and 
$(\Psi, \lambda) \in \mathcal V$ we have 
\[\mathcal C^{\star}(\Psi, \lambda) \leq  \mathcal L(\rho,J, \Psi, \lambda) = \mathcal C(\rho, J).\]
The following lemma is therefore immediate.
\begin{lemma}
\label{lem:duality-principle}
Suppose $(\rho^{\star}, J^{\star}) \in \mathcal U$, $(\Psi^{\star}, \lambda^{\star}) \in \mathcal V$. If the constraints~\eqref{eq:constraints} are satisfied for $(\rho^{\star}, J^{\star})$ and $\mathcal C^{\star}(\Psi^{\star}, \lambda^{\star}) = \mathcal C(\rho^{\star}, J^{\star})$, then $(\rho^{\star}, J^{\star})$ is optimal for Problem~\ref{prob:ergodic_minimization_with_current}.
\end{lemma}

\begin{lemma}
\label{lem:dual-functional}
Suppose $A = 0$. The dual cost functional $\mathcal C^{\star}$ can be expressed as
\begin{equation}
 \mathcal C^{\star}(\Psi, \lambda) = \begin{cases}- \lambda \quad & \mbox{if} \ I[\Psi, \lambda](x) \geq 0 \ \mbox{for all $x \in M$}, \\
                                      -\infty & \mbox{if} \ I[\Psi, \lambda](x) < 0 \ \mbox{for some $x \in M$},
                                     \end{cases}
\end{equation}
where
\begin{equation}
 I[\Psi, \lambda]:= V - \half \| \nabla \Psi\|^2  - \langle b, \nabla \Psi \rangle - \half \Delta \Psi + \lambda.
\end{equation}
\end{lemma}

\begin{proof}
We have 
\[ \mathcal L(\rho, J, \Psi, \lambda) = \int_M \left[ \left\{ V + \half \left\| \frac 1 \rho (J + \half \nabla \rho) - b \right\|^2 + \lambda \right\} \rho  - \langle J, \nabla \Psi \rangle \right] \ d x - \lambda.\]
For fixed  $\rho$ the choice $J^{\star} = -\half \nabla \rho  + \rho( b + \nabla \Psi)$ determines the pointwise minimum of the integrand. For this choice of $J$, we obtain
\begin{align*} \mathcal L(\rho, J^{\star}, \Psi, \lambda) & = \int_M \left[ \left\{ V + \half \|\nabla \Psi\|^2 +  \lambda \right \} \rho + \langle (\half \nabla \rho - \rho(b + \nabla \Psi)), \nabla \Psi \rangle \right] \ d x- \lambda \\
 & = \int_M \left\{ V - \half \| \nabla \Psi\|^2 + \half \langle  \nabla (\log \rho), \nabla \Psi \rangle - \langle b, \nabla \Psi \rangle + \lambda \right\} \rho \ d x - \lambda \\
 & = \int_M\left[  \left\{ V - \half \| \nabla \Psi\|^2  - \langle b, \nabla \Psi \rangle + \lambda  \right\} \rho  + \half \langle  \nabla \rho, \nabla \Psi \rangle \right]  \ d x -  \lambda \\
 & = \int_M  \left\{ V - \half \| \nabla \Psi\|^2  - \langle b, \nabla \Psi \rangle - \half \Delta \Psi  + \lambda \right\} \rho    \ d x - \lambda.
\end{align*}
Now if $I(\Psi, \lambda)  < 0$ on some open subset of $M$, then $\lambda(\rho, J^{\star}, \Psi, \lambda)$ can assume arbitrarily large negative values by letting $\rho$ approach a Dirac delta peak centered within that subset, so that in this case $\mathcal L(\rho, J^{\star}, \Psi, \lambda) \rightarrow -\infty$. If $I(\Psi, \lambda) \geq 0$ on $M$, then $\mathcal L(\rho, J^{\star}, \Psi, \lambda)$ is minimized by taking $\rho^{\star} = 0$.
\end{proof}

\begin{lemma}
\label{lem:sufficient_condition}
Suppose $A = 0$. Suppose $(\Psi, \lambda) \in \mathcal V$ satisfies $I[\Psi,\lambda] = 0$, i.e.
\[ \half \Delta \Psi + \half \| \nabla \Psi\|^2 + \langle b, \nabla \Psi \rangle - V = \lambda.\]
Let $u =\nabla \Psi$ and let $\rho > 0$ be the unique solution to the Fokker-Planck equation $L_u^{\star} \rho = 0$ with $\int_M \rho \ d x = 1$. Define $J = - \half \nabla \rho + \rho (b + \nabla \Psi)$.
Then $(\rho,J)$ solves Problem~\ref{prob:ergodic_minimization_with_current} , and $C(\rho,J) = -\lambda$.\end{lemma}

\begin{proof}
By Lemma~\ref{lem:fokker_planck_equivalence}, $\div J = 0$, so that $(\rho,J)$ satisfy the constraints~\eqref{eq:constraints}. Furthermore
\begin{align*} \mathcal C(\rho,J) & = \int_M \left\{ V + \half \left\| \frac 1 \rho (J + \half \nabla \rho) - b \right\|^2  \right\} \rho \ d x = \int_M \left\{ 2 V - \half \Delta \Psi - \langle b, \nabla \Psi \rangle \right\} \rho  \ d x + \lambda \\
& = \int_M \left[  2 V \rho  + (\div(b \rho) - \half \Delta \rho ) \Psi \right]\ d x + \lambda \\
& = \int_M \left[ 2 V \rho - (L^{\star}_u \rho) \Psi  - \div(\rho u) \Psi \right] \ d x + \lambda  = \int_M  \left\{ 2 V  + \|\nabla \Psi\|^2 \right\} \rho  \ d x + \lambda = 2 \mathcal C(\rho,J) + \lambda.
\end{align*}
Consequently, $\mathcal C(\rho, J) = - \lambda$. Also, for this choice of $\Psi$ and $\lambda$ we have by Lemma~\ref{lem:dual-functional} that $\mathcal C^{\star}(\Psi, \lambda) = - \lambda$. The optimality of $(\rho, J)$ now follows from Lemma~\ref{lem:duality-principle}.
\end{proof}

The following proposition is now a direct consequence of Remark~\ref{rem:remove_A}, Lemma~\ref{lem:sufficient_condition} and a brief computation.

\begin{proposition}[Gauge invariant Hamilton-Jacobi-Bellman equation -- sufficiency]
\label{prop:nonlinear_hjb}
Suppose $(\Psi, \lambda) \in \mathcal V$ satisfies
\begin{equation}
 \label{eq:nonlinear_hjb}
\half \div( \nabla \Psi - A)  + \half \| \nabla \Psi - A \|^2 + \langle b, (\nabla \Psi - A)\rangle - V  = \lambda.
\end{equation}
Let $u = \nabla \Psi - A$, let $\rho$ denote the unique solution to $L_{u}^{\star} \rho = 0$ with $\int_M \rho \ d x = 1$, and let $J = - \half \nabla \rho + \rho(b + \nabla \Psi - u)$. Then $(\rho,J)$ is optimal for Problem~\ref{prob:ergodic_minimization_with_current} and $\mathcal C(\rho,J)=- \lambda$.
\end{proposition}

Equation~\eqref{eq:nonlinear_hjb}, the \emph{gauge-invariant Hamilton-Jacobi-Bellman equation} was already discussed in our related physics paper, see \cite[equation (10)]{Chernyak2013}. As remarked there, the equation may be linearised by an exponential transformation. The verification of this result is straightforward.

\begin{corollary}[Linear gauge invariant Hamilton-Jacobi-Bellman equation -- sufficiency]
\label{cor:linear_hjb}
Suppose $(\psi, \lambda) \in C^{\infty}(M) \times \R$, satisfy $\psi > 0$ and 
\begin{equation}
 \label{eq:linear_hjb} 
 \half \div(\nabla \psi - A \psi ) + \langle (b-A), (\nabla \psi - A \psi) \rangle -( V + \half \|A\|^2) \psi = \lambda \psi.
\end{equation}
Then~\eqref{eq:nonlinear_hjb} is satisfied for $\Psi = \log \psi$, so that the conclusion of Proposition~\ref{prop:nonlinear_hjb} applies. In terms of $\psi$ the optimal control and associated current are given by
\begin{equation}
\label{eq:optimal-control-current} u = \nabla \log \psi - A \quad \mbox{and} \quad J = -\half \nabla \rho + \rho(b + \nabla \log \psi -A).
\end{equation}
\end{corollary}

Essentially~\eqref{eq:linear_hjb} is an eigenvalue problem for a non-degenerate second order elliptic differential equation on a compact manifold, for which it is well established that a solution $(\psi, \lambda)$ exists such that $\psi$ is smooth and for which $\psi > 0$; see \cite{Warner1983}.

\subsection{Necessary condition for optimality}
\label{sec:necessary}

In order to obtain necessary conditions for optimality of $(\rho, J)$ for Problem~\ref{prob:ergodic_minimization_with_current} we will relax the problem to an optimization problem over Sobolev spaces. In particular, we will rephrase it as the following abstract optimization problem.
Let $X$ and $Z$ be Banach spaces and let $\mathcal U$ be an open set in $X$. Let $\mathcal C : \mathcal U \subset X \rightarrow \R$ and $\mathcal H : \mathcal U \subset X \rightarrow Z$.

\begin{problem}
\label{prob:optimization_in_banach_spaces}
Minimize $\mathcal C(x)$ over $\mathcal U$ subject to the constraint $\mathcal H(x) = 0$.
\end{problem}

The \emph{Fr\'echet derivative} \cite{Luenberger1969} of a mapping $T : D \subset X \rightarrow Y$ in $x \in D$ will be denoted by $T'(x) \in L(X;Y)$.
We will need the following notion.
\begin{definition}[Regular point]
\label{def:regular-point}
 Let $T$ be a continuously Fr\'echet differentiable function from an open set $D$ in a Banach space $X$ into a Banach space $Y$. If $x_0 \in D$ is such that $T'(x_0)$ maps $X$ onto $Y$, then the point $x_0$ is said to be a \emph{regular point} of the transformation $T$.
\end{definition}

For the abstract Problem~\ref{prob:optimization_in_banach_spaces} the following necessary condition holds for a local extremum \cite[Theorem 9.3.1]{Luenberger1969}.
\begin{lemma}[Lagrange multiplier necessary conditions]
\label{lem:lagrange_multipliers}
Suppose $\mathcal C$ and $\mathcal H$ are continuously Fr\'echet differentiable on $\mathcal U$. If $\mathcal C$ has a local extremum under the constraint $\mathcal H(x) = 0$ at the regular point $x_0 \in \mathcal U$, then there exists an element $z^{\star}_0 \in Z^{\star}$ such that 
\[ \mathcal C'(x_0) + \langle \mathcal H'(x_0), z_0^{\star} \rangle = 0.\]
Here $\langle \cdot, \cdot \rangle$ denotes the pairing between $Z$ and $Z^{\star}$.
\end{lemma}

We will define Sobolev spaces of functions and  vector fields as follows. 
For $k \in \N \cup \{ 0 \}$
let $H^k(M)$ be the completion of $C^{\infty}(M)$ with respect to the norm
\[ \| \varphi\|_{H^k(M)}^2 := \sum_{l=0}^k  \int_M \| \nabla^l \varphi \|^2 \ d x,\]
equipped with Hilbert space structure induced by $\| \cdot\|_{H^k(M)}$.
Similarly, let $H^k_{\mathfrak X}(M)$ be the completion of $\mathfrak X(M)$ with respect to the norm
\[ \| \xi \|_{H^k_{\mathfrak X}(M)}^2 := \sum_{l=0}^k \int_M  \| \nabla^l \xi \|^2 \ d x,\]
equipped similarly with Hilbert space structure.
Note that the Laplace-Beltrami operator $\Delta$ is a bounded mapping $\Delta: H^k(M) \rightarrow H^{k-2}(M)$. Also $\div: H^{k}_{\mathfrak X}(M) \rightarrow H^{k-1}(M)$ and $\nabla : H^{k}(M) \rightarrow H^{k-1}_{\mathfrak X}(M)$ are bounded linear mappings.
Recall Sobolev's Lemma \cite[Proposition 4.3.3]{Taylor1996}.

\begin{lemma}[Sobolev embedding]
\label{lem:sobolev_embedding}
Suppose $\varphi \in H^k(M)$ and suppose $m \in \N \cup \{0\}$ satisfies $k > n / 2 + m$. Then $\varphi \in C^m(M)$.
\end{lemma}

For $k \in \N$, $k \geq n/2+1$, define spaces as follows. Let $X^k := H^{k+1}(M) \times H^k_{\mathfrak{X}}(M)$.
Since $k \geq n/2+1$, we have by the Sobolev Lemma that $(\rho, J) \in X^k$ satisfies $\rho \in C(M)$.
Furthermore define 
\begin{align*} H^k_+(M) & := \left\{ \rho \in H^{k}(M) : \rho > 0 \ \mbox{on} \ M \right\}, \\
 \mathcal U^k & := H^{k+1}_+(M) \times H^{k}_{\mathfrak X}(M), \\
 Y^k & :=  \left\{ \psi \in H^{k-1}(M): \psi = \div \xi \ \mbox{for some} \ \xi \in H^k_{\mathfrak X}(M) \right\}, \quad & \mbox{and} \\
 Z^k & := Y^k \times \R.
\end{align*}
The condition $\rho > 0$ defines an open subset $\mathcal U^k \subset X^k$.



\begin{lemma}
Suppose $k \geq n/2+1$. The mapping $\mathcal C(\rho, J)$, as given by~\eqref{eq:costfunction_intermsof_J_and_rho}, may be continuously extended to a mapping $\mathcal C : \mathcal U^k \rightarrow \R$. Moreover, the mapping $\mathcal C$ is continuously differentiable on $\mathcal U^k$ with Fr\'echet derivative $\mathcal C'(\rho, J) \in L(X^k; \R)$ given for $(\rho, J) \in \mathcal U^k$ by
\begin{align*} \mathcal C'(\rho,J) : (\zeta, G) & \mapsto \int_M \left\{ V + \frac  1{ 2} \|b\|^2 - \frac 1 { 2  \rho^2} \| J + \half \nabla \rho \|^2 + \half \div \left(  b - \frac 1 {\rho} \left( J + \half \nabla \rho \right) \right) \right\} \zeta  d x \\
& \quad + \int_M \left \langle \left( - b + \frac 1 {\rho} ( J + \half \nabla \rho) \right) + A , G \right \rangle  \ d x.
\end{align*}
\end{lemma}

\begin{proof}
We compute the directional (Gateaux) derivative $\mathcal C'(\rho,u)$ to be the linear functional on $X$ given by
\begin{align*} & (\zeta, G)\mapsto  \\
& \int_M \left\{ V + \frac 1 {2} \left\|-b + \frac{1}{\rho} (J + \half \nabla \rho ) \right\|^2 \right\} \zeta
+  \left \langle - b + \frac 1 { \rho} (J + \half \nabla \rho ), - \frac 1 {\rho^2} \left( J + \half \nabla \rho\right) \zeta + \frac 1 {2 \rho} \nabla \zeta \right \rangle \ \rho \, d x \\
& \quad + \int_M \left\{  \left \langle - b + \frac 1 {\rho} ( J + \half \nabla \rho), \frac 1 {\rho} G \right \rangle \rho + \langle A, G \rangle  \right\} \, d x,
\end{align*}
This is after rearranging, and partial integration of the term containing $\nabla \zeta$, equal to the stated expression.
The derivative $\mathcal C'(\rho, u)$ is a bounded functional on $X^k$ since $V, 1 /\rho, d \rho, b, J$ and $A$ are bounded on $M$.  Since the derivative depends continuously on $(\rho,u)$, it is in fact the Fr\'echet derivative of $\mathcal C$.
\end{proof}

We define the constraint mapping $\mathcal H : \mathcal U^k \rightarrow Z^k$ as
\begin{equation} \label{eq:constraint-mapping} \mathcal H(\rho, J) := \left( \div J, \int_M \rho \ d x - 1 \right), \quad (\rho, J) \in \mathcal U^k. \end{equation}
The following lemma is now immediate.
\begin{lemma}
The mapping $\mathcal H$ is continuously differentiable on $X^k$, with Fr\'echet derivative $\mathcal H'(\rho, J) \in L(X^k; Z^k)$ given for $(\rho, J) \in X^k$ by
\[ \mathcal H'(\rho,J) : (\zeta, G) \mapsto \left( \div G, \int_M \zeta \ d x \right), \quad (\zeta, v) \in X^k.\]
\end{lemma}

Every $(\rho, J) \in X^k$ is regular for $\mathcal H$, thanks to our choice of the function space $Z^k$.
\begin{lemma}
Any $(\rho, J) \in X^k$ is a regular point of $\mathcal H$ (in the sense of Definition~\ref{def:regular-point}).
\end{lemma}
\begin{proof}
Let $(\Psi, \kappa) \in Z^k = Y^k \times \R$. In particular there exists a $\xi \in H^{k}_{\mathfrak X}(M)$ such that $\Psi = \div \xi$. We may pick $G = \xi$, and $\zeta$ a constant function such that $\int_M \zeta \ d x = \kappa$. Then $\mathcal H'(\rho, J) (\zeta, G) = (\Psi, \kappa)$, showing that $\mathcal H'(\rho,J)$ is onto.
\end{proof}

In order to apply the abstract Lagrange multiplier theorem (Lemma~\ref{lem:lagrange_multipliers}) in a useful manner, we need to give interpretation to the dual spaces $(Z^k)^{\star}$, and in particular to $(Y^k)^{\star}$.
Recall that the spaces $(H^l(M))^{\star}$, for $l \in \N \cup \{0\}$ may be canonically identified through the $L^2(M)$-inner product with spaces of distributions, denoted by $H^{-l}(M)$ \cite[Proposition 4.3.2]{Taylor1996}. In other words, if $z \in (H^l(M))^{\star}$, then there exists a distribution $\Phi \in H^{-l}(M)$ such that $z(\Psi) = \int_M \Phi \Psi \ d x$. Now in case $\Psi \in Y^k$, i.e. $\Psi = \div \xi$ for some $\xi \in H^k_{\mathfrak X}(M)$, then $z(\Psi) = \int_M \Phi \div \xi \ d x = - \int_M \langle \nabla \Phi, \xi \rangle \ d x$ for some $\Phi \in H^{-(k-1)}(M)$. Therefore the choice of $\Phi \in H^{-(k-1)}(M)$ representing $z \in (Y^k)^{\star}$ is fixed up to the addition of a ``constant'' distribution: if $\nabla \gamma \in H^{-(k-1)}(M)$ vanishes in a weak sense, then $\int_M (\Phi + \gamma) \div \xi \ d x = \int_M \Phi \div \xi \ d x$. We summarize this in the following lemma.

\begin{lemma}
\label{lem:Y_dual}
$(Y^k)^{\star} \cong H^{-(k-1)}(M) / \left\{ \gamma \in H^{-(k-1)}(M) : \nabla \gamma = 0 \right\}$.
\end{lemma}

We may now apply Lemma~\ref{lem:lagrange_multipliers} to obtain the following preliminary result. 

\begin{proposition}[Gauge invariant Hamilton-Jacobi-Bellman equation -- necessity]
\label{prop:nonlinear-hjb-necessary}
Suppose $(\rho, J) \in C^{\infty} \times \mathfrak X(M)$ is a local extremum of $\mathcal C$, defined by~\eqref{eq:costfunction_intermsof_J_and_rho}, under the constraint that $\mathcal H(\rho, J) = 0$. Then there exists $\Psi \in C^{\infty}(M)$ and $\lambda \in \R$ satisfying~\eqref{eq:nonlinear_hjb} and such that $J = - \half \nabla \rho + \rho (b + \nabla \Psi - A)$. The corresponding control field $u \in \mathfrak X(M)$ is given by $u = \nabla \Psi - A$.
\end{proposition}

\begin{proof}
Let $k \geq n/2 + 1$. 
Then $(\rho, J) \in \mathcal U^k$. By Lemma~\ref{lem:lagrange_multipliers}, there exists an element $(\Psi, \lambda) \in H^{-(k-1)}(M) \times \R$ such that the following equations hold.
\begin{align*}
 V + \frac  1{ 2} \|b\|^2 - \frac 1 { 2  \rho^2} \| J + \half \nabla \rho \|^2 + \frac 1 {2 } \div \left(  b - \frac 1 {\rho} \left( J + \half \nabla \rho \right) \right) + \lambda & = 0, \quad \mbox{and} \\
 \left( - b + \frac 1 {\rho} ( J + \half \nabla \rho) \right) + A - \nabla \Psi & = 0.
\end{align*}
(By Lemma~\ref{lem:Y_dual} $\Psi$ is defined up to a constant). Substituting the second equation into the first, and making some rearrangements, gives the sytem~\eqref{eq:nonlinear_hjb}. Then $\Psi \in C^{\infty}(M)$ as a result of the expression for $J$. The expression for $u$ is an immediate result of~\eqref{eq:u_intermsof_J}.
\end{proof}

By letting $\psi = \exp(\Psi)$, we immediately obtain the following result.

\begin{corollary}[Linear gauge invariant HJB equation -- necessity]
\label{cor:linear-hjb-necessary}
Suppose $(\rho, J) \in \mathcal U$ is a local extremum of $\mathcal C$, defined by~\eqref{eq:costfunction_intermsof_J_and_rho}, under the constraints~\eqref{eq:constraints}. Then there exists a $\psi \in C^{\infty}(M)$, $\psi > 0$ on $M$, and $\lambda \in \R$ such that~\eqref{eq:linear_hjb} holds.
Furthermore $\rho \in C^{\infty}(M)$, $J \in \mathfrak X(M)$ and the associated control field $u \in \mathfrak X(M)$ are  related by~\eqref{eq:optimal-control-current}.\end{corollary}

%
%

Since the problem considered in Corollary~\ref{cor:linear_hjb} is a relaxed version of Problem~\ref{prob:ergodic_minimization_with_current}, and smoothness of $J$ and $\rho$ is established in the relaxed case, we immediately have the following corollary.

\begin{corollary}
Suppose $(\rho, u)$ is a solution of Problem~\ref{prob:ergodic_minimization}, or equivalently that $(\rho, J)$ a solution of Problem~\ref{prob:ergodic_minimization_with_current}. Then the results of Corollary~\ref{cor:linear-hjb-necessary} hold.
\end{corollary}

\begin{remark}
\label{rem:donsker_varadhan}
Combining Corollaries~\ref{cor:linear_hjb} and~\ref{cor:linear-hjb-necessary}, taking $A = 0$, gives a variational characterization of the principal eigenvalue of an elliptic differential operator $L = \half \Delta + b \nabla - V$, given by
\[ \lambda^{\star} = - \inf_{\substack{(\rho, J) \in C^{\infty} \times \mathfrak X(M) \\ \int_M \rho \ d x = 1 \ \mbox{and} \ \div J = 0}} \mathcal C(\rho, J).\]
This result may be compared to \cite{DonskerVaradhan1975-I}, in which a variational principle is derived for the maximal eigenvalue of an operator $L$ satisfying a maximum principle.
\end{remark}

\subsection{Reversible solution}
\label{sec:reversible_solution}
In the reversible case we can represent the optimally controlled invariant measure in terms of $\psi$ and the uncontrolled invariant measure.
Let $(T(t))_{t \geq 0}$ denote the transition semigroup of a diffusion on $M$. The diffusion is said to be \emph{reversible} if there exists a Borel measure $\nu(dx)$ such that
\[ \int_M (T(t) f)(x) g(x) \ \nu \ d x = \int_M f(x) (T(t) g)(x) \ \nu \ d x \quad \mbox{for all} \ f, g \in C(M), t \geq 0.\]
Other equivalent terminology is that the Markov process is \emph{symmetrizable} or that the invariant measure satisfies \emph{detailed balance} \cite[Section V.4]{IkedaWatanabe1989}.
In case a diffusion is reversible with respect to a measure $\nu$, this measure is an invariant measure for the diffusion.

The following results hold for any control field $u \in \mathfrak X(M)$.
\begin{lemma}
\label{lem:reversible_equivalence}
Let $X$ denote a diffusion with generator given by $L h = \half \Delta h + \langle (b + u), \nabla h \rangle$, $h \in C^2(M)$. The following are equivalent. 
\begin{itemize}
 \item[(i)] $X$ is reversible, with invariant density $\rho_u = \exp(-U)$ for some $U \in C^{\infty}(M)$;
 \item[(ii)] $b + u = - \half \nabla U$ for some $U \in C^{\infty}$;
 \item[(iii)] The long term current density $J$, given by~\eqref{eq:long_term_current_density}, vanishes.
\end{itemize}
\end{lemma}

\begin{proof}
The equivalence of (i) and (ii) is well known, see e.g. \cite[Theorem V.4.6]{IkedaWatanabe1989}.
The equivalence of (ii) and (iii) is then immediate from~\eqref{eq:long_term_current_density}.
\end{proof}

\begin{proposition}
\label{prop:reversible_solution}
Let $\rho_0$ and $J_0 = - \half \nabla \rho_0 + \rho_0 b$ denote the density and current corresponding to the uncontrolled dynamics~\eqref{eq:uncontrolled}.
The following are equivalent.
\begin{itemize}
\item[(i)] The diffusion corresponding to the optimal control $u$ is reversible, with density $\rho = \psi^{2} \rho_0$, where $\psi$ is as in Corollary~\ref{cor:linear_hjb} (normalized such that $\int_M \psi^{2} \rho_0 \ dx = 1$);
\item[(ii)] $J_0 = \rho_0 A$;
\item[(iii)] $b = - \half \nabla U + A$, for some $U \in C^{\infty}(M)$.
\end{itemize}

In particular, if the uncontrolled diffusion is reversible and $A = 0$, then the controlled difussion is reversible and the density admits the expression given under (i).
\end{proposition}


\begin{proof}
Setting $\rho = \psi^{2} \rho_0$, we have 
\begin{align*} J & = -\half \nabla \rho + \rho \left(b -  A + \nabla \log \psi \right)
= - \half \psi^2 \nabla \rho_0 - \rho_0 \psi \nabla \psi + \psi^2 \rho_0 \left( b -  A + \nabla \log \psi \right) \\
& = - \half \psi^2 \nabla \rho_0 + \psi^2 \rho_0 (b -  A)
= \psi^2 (J_0 -  \rho_0 A),
\end{align*}
which establishes the equivalence of (i) and (ii).
Representing the density $\rho_0$ by $\exp(-U)$ and using~\eqref{eq:long_term_current_density} with $u = 0$ gives the equivalence between (ii) and (iii).
\end{proof}

\subsection{Gauge invariance}
\label{sec:gauge_invariance}
For a special choice of $A$, the solution of Problem~\ref{prob:ergodic_minimization_with_current} may be related to the solution corresponding to $A = 0$ in a simple way.

\begin{proposition}
Let $A_0 \in \mathfrak X(M)$. For $\varphi \in C^{\infty}(M)$ and $A = A_0 + \nabla \varphi$ let $(\rho_{\varphi}, J_{\varphi})$ denote the solution of Problem~\ref{prob:ergodic_minimization_with_current} with corresponding solutions $\psi_{\varphi}$ and $u_{\varphi}$ of~\eqref{eq:linear_hjb}. Then, for $\varphi \in C^{\infty}(M)$,
\begin{equation} \label{eq:gauge_invariance} \rho_{\varphi} = \rho, \quad u_{\varphi} = u, \quad \psi_{\varphi} = \exp(\varphi) \psi, \quad \mbox{and} \quad J_{\varphi} = J,
\end{equation}
where $\rho$, $J$, $\psi$ and $u$ denote the solution of Problem~\ref{prob:ergodic_minimization_with_current} and~\eqref{eq:linear_hjb} corresponding to $A = A_0$.
\end{proposition}
\begin{proof}
This is a matter of straightforward computation.
\end{proof}

In other words, the solution of Problem~\ref{prob:ergodic_minimization_with_current} depends (essentially) only on the equivalence class of $A$, under the equivalence relation $A \sim B$ if and only if $A = B + \nabla \varphi$ for some $\varphi \in C^{\infty}(M)$.

\begin{remark}
A standard way in physics to obtain gauge
invariant differential operators is to replace the gradients with `long' derivatives. 
This phenomenon also occurs in the linear Bellman equation.
In differential form notation (interpreting $A$ as differential form; see Remark~\ref{rem:A_actually_diff_form}) the linear operator on the left hand side of~\eqref{eq:linear_hjb} can be written as 
\[ H \psi  = \half \star \left( d -  A \wedge \right) \star \left( d -  A \wedge \right)\psi  + \star \left( f \star (d -  A \wedge)\right) \psi  -  V \psi.\]
The operator $\psi \mapsto d \psi -  A \wedge \psi$ is called a \emph{`long' derivative} operator.
This result is easily verified using the standard relations
\[ \star \left( \alpha \wedge \star \beta \right) = \langle \alpha, \beta \rangle, \quad \star d \star d\phi = \Delta \phi, \quad \star d \star \alpha = - \delta \alpha,\]
for $\alpha, \beta \in \Lambda^1(M)$ and $\phi \in C^{\infty}(M)$.

See also the directly related expression \cite[Equation (31)]{BaratoChetrite2015}, where a similar twisted second order generator occurs in relation to a tilting of a Markov process to accommodate for currents.
\end{remark}

\section{Fixed density}
\label{sec:fixed_density}
In this section we consider the problem of fixing the density function $\rho$, and finding a force $u$ that obtains this density function, at minimum cost. 
Let $\rho \in C^{\infty}(M)$ be fixed, with $\rho > 0$ on $M$ and $\int_M \rho \ d x = 1$. Then for some constant $c_{\rho}$ we have 
\[ \mathcal C(\rho, u) = c_{\rho} + \int_M \left\{ \frac 1 {2 } \|u\|^2 + \langle  A, u \rangle \right\} \rho \ d x. \] 
Therefore we will consider the following problem.
\begin{problem} 
\label{prob:fixed_density}
Minimize $\mathcal C(u)$ over $\mathfrak X(M)$, 
subject to the constraint $L_u^{\star} \rho = 0$, 
where $\mathcal C(u)$ is defined by
\[ \mathcal C(u) = \int_M \left\{ \frac 1 {2} \|u\|^2 + \langle  A, u \rangle \right\} \rho \ d x.\]
\end{problem}

The corresponding problem in terms of the curent density is the following. As for $\mathcal C(u)$, terms that do not depend on $J$ are eliminated from the cost functional.
\begin{problem}
\label{prob:fixed_density_in_terms_of_J}
Minimize $\mathcal C(J)$ over $\mathfrak X(M)$, subject to the constraint $\div J = 0$, where $\mathcal C(J)$ is defined by
\begin{equation} \label{eq:cost_fixed_density_in_J} \mathcal C(J) = \int_M \left( \left \langle - b + \half \nabla \log \rho +  A, J \right \rangle + \frac 1 { 2 \rho} \|J\|^2 \right) \ d x. \end{equation}
\end{problem}

By a standard variational argument we obtain the following result, which states that  a relaxed version of Problem~\ref{prob:fixed_density_in_terms_of_J} may be transformed into an elliptic PDE. Essentially, this is obtained through variation of the Lagrangian functional 
\[\mathcal L(J, \Phi) = \mathcal C(J) + \int_M \Phi \div  J \ d x.\]

\begin{theorem}
\label{thm:solution_fixed_density}
Suppose $J \in \mathfrak X(M)$ is a local extremum of $\mathcal C(J)$ given by~\eqref{eq:cost_fixed_density_in_J} under the constraint that $\div J = 0$. Then there exists a $\Phi \in C^{\infty}(M)$ such that
\begin{equation} \label{eq:linear_pde_fixed_density} \Delta \Phi + \left \langle \nabla \log  \rho, \nabla \Phi \right\rangle =  \frac 1 {\rho} \left( \half \Delta \rho  - \div (\rho (b -   A) ) \right).\end{equation}
For this $\Phi$, $J$ is given by 
\[ J = \rho f - \half  \nabla \rho + \rho  (\nabla \Phi -  A),\]
and the corresponding control field $u$ is given by 
\[ u =  \nabla \Phi - A.\]
\end{theorem}

%
%

\begin{remark}
It appears that the result of Theorem~\ref{thm:solution_fixed_density} can be thought of as a gauge invariant extension of the characterization~\eqref{eq:gartner} of the Donsker-Varadhan functional by G\"artner \cite{Gartner1977}. Taking $A = 0$ and comparing with~\eqref{eq:gartner} we may rephrase the Donsker-Varadhan rate functional~\eqref{eq:donsker-varadhan-rate-function}, as
\[ I(\mu) = \inf_{\substack{u \in \mathfrak X(M)\\ L^{\star}_u \rho = 0}} \half \int_M \| u\|^2\  \rho \ dx \]
for $\mu$ absolutely continuous with density $\rho$. (Note that strictly speaking, we need to be careful since we have only obtained Theorem~\ref{thm:solution_fixed_density} under the condition that a local minimum exists.)
\end{remark}

\begin{remark}[Solution in the reversible case]
\label{rem:solution_fixed_density_special_case}
If $b -  A = - \half \nabla  U$ for some $U \in C^{\infty}(M)$, it may be checked that $\Phi =  \frac 1 {2 } \left( \ln \rho + U \right)$ solves~\eqref{eq:linear_pde_fixed_density}, so that the optimal control field
$u = \half \nabla (\ln \rho) - b$. In other words, the optimal way to obtain a particular density function $\rho$ if $b -  A$ is in `gradient form' is by using a control $u$ so that the resulting force field $b + u$ is again in gradient form, $b + u = \half \nabla (\ln \rho)$, resulting in reversible dynamics; see also Section~\ref{sec:reversible_solution}.
\end{remark}

\begin{example}[Circle]
On $S^1$ every differential 1-form $\beta$, and in particular $\beta = b - A$,  may be written as $\beta = -\half \ d U + \half k \ d \theta$, where $\theta$ represents the polar coordinate function, $U \in C^{\infty}(S^1)$ and $k \in \R$; see e.g. \cite[Example 4.14]{Warner1983}. Equation~\eqref{eq:linear_pde_fixed_density} then reads
\[ \Phi''(\theta) + \frac 1 {\rho} \Phi'(\theta) \rho'(\theta) = \frac 1 {2  \rho} \left( \rho''(\theta) - \frac{d}{d\theta}\left(- \rho(\theta) U'(\theta) + k \rho \right) \right).\]
Based on Remark~\ref{rem:solution_fixed_density_special_case}, we try a solution of the form
$\Phi = \frac 1 {2 } (\ln \rho + U - \varphi)$. Inserting this into the differential equation, we obtain for $\varphi$ the equation
\[ \varphi''(\theta) + \gamma(\theta) \varphi'(\theta) = k \gamma(\theta),\]
where $\gamma(\theta) = \frac{d}{d\theta} \ln \rho(\theta) = \frac 1{\rho(\theta)} \rho'(\theta)$.

Up to an arbitrary additive constant (which we put to zero), there exists a unique periodic solution $\varphi$ to this differential equation, given by
\[ \varphi(\theta) = \frac{k \left( \theta \int_0^{2 \pi} \rho(\xi)^{-1} \ d \xi - 2 \pi \int_0^{\theta} \rho(\xi)^{-1} \ d \xi \right)}{\int_0^{2 \pi} \rho(\xi)^{-1} \ d \xi }, \quad \theta \in [0, 2 \pi]. \]
\end{example}

\begin{remark}[Gauge invariance]
As in Section~\ref{sec:gauge_invariance}, it is straightforward to check that a solution to Problem~\ref{prob:fixed_density_in_terms_of_J} for $A_{\varphi} = A_0 + \nabla  \varphi$ is given by
\[ \Phi_{\varphi} = \Phi + \varphi, \quad \rho_{\varphi} = \rho, \quad u_{\varphi} = u, \quad J_{\varphi} = J,\]
in terms of the solution $(\Phi, \rho, u, J)$ corresponding to the gauge field $A_0$.
\end{remark}

%
%
%
\section{Fixed current density}
\label{sec:fixed_current_density}
In this section we approach the problem of minimizing the average cost, under the constraint that $J$ is fixed. In light of the remark just below~\eqref{eq:long_term_current_density}, it will be necessary to demand that $\div J = 0$, otherwise we will not be able to obtain a solution. 
Hence in the remainder of this section let $J \in \mathfrak X(M)$ satisfying $\div J = 0$ be fixed. 
By~\eqref{eq:u_intermsof_J}, we may express $u$ in terms of $J$ and $\rho$ by $u = -b + \frac 1 {\rho} \left( J + \half \nabla \rho \right)$. Note that by Lemma~\ref{lem:fokker_planck_equivalence}, the Fokker-Planck equation~\eqref{eq:fokker_planck} is satisfied. This leads to the following problem.



\begin{problem}
\label{prob:fixed_current}
Minimize $\mathcal C(\rho)$ subject to the constraint $\int_M \rho \ d x = 1$, where 
\begin{equation} \label{eq:cost_fixed_current} \mathcal C(\rho) = \int_M \left( V + \frac 1{2 } \left\| -b  + \frac 1{\rho} \left( J + \half \nabla \rho \right) \right\|^2  \right) \rho \ d x.\end{equation}
\end{problem}

\begin{remark}
The constraint $\rho \geq 0$ on $M$ does not need to be enforced, since if we find $\rho$ solving Problem~\ref{prob:fixed_current} without this constraint, we may compute $u$ by~\eqref{eq:u_intermsof_J}. Then $\rho$ satisfies $L_u^{\star} \rho = - \div J = 0$ by Lemma~\ref{lem:fokker_planck_equivalence}, so by Proposition~\ref{prop:existence_uniqueness_inv_measure} and the constraint $\int_M \rho \ d x = 1$, it follows that $\rho > 0$ on $M$.
\end{remark}

\begin{remark}
Note that by Lemma~\ref{lem:ergodic_gaugefield_current}, the contribution of $A$ is determined once we fix $J$. Therefore we may put $A = 0$ in the current optimization problem.
\end{remark}

Necessary conditions for the solution of Problem~\ref{prob:fixed_current} may be obtained rigorously in a similar manner as in Section~\ref{sec:hjb}, to obtain the following result.

\begin{theorem}
\label{thm:solution_fixed_current}
Suppose $\rho \in C^{\infty}(M)$ minimizes $\mathcal C(\rho)$ given by~\eqref{eq:cost_fixed_current} under the constraint that $\int_M \rho \ d x = 1$. Then there exists a $\mu \in \R$ such that
\begin{equation}
\label{eq:yermakov}
\half \Delta \phi - (W +  \mu) \phi = - \frac{\|J\|^2}{2 \phi^3},
\end{equation}
holds, where $\phi = \sqrt \rho$ and $W =  V + \half \|b\|^2 + \half \div b$. 
\end{theorem}

\begin{remark}
Equation~\eqref{eq:yermakov} is known (at least in the one-dimensional case) as \emph{Yermakov's equation} \cite{Polyanin2003}.
\end{remark}

Instead of proving Theorem~\ref{thm:solution_fixed_current} rigorously (which may be done analogously to Section~\ref{sec:hjb}) we provide an informal derivation, which we hope provides more insight to the reader. We introduce the Lagrangian $\mathcal L : C^{\infty}(M) \times \R \rightarrow \R$ by
\begin{align*}
 \mathcal L(\rho, \mu) 
& = \int_M \left( V + \frac 1{2 } \left\|-b + \frac 1{\rho} \left( J + \half \nabla  \rho \right) \right\|^2  \right) \rho \ d x + \mu \left( \int_M \rho \ d x - 1 \right).
\end{align*}

Varying $\mathcal L(\rho,\mu)$ with respect to $\rho$ in the direction $\zeta \in C^{\infty}(M)$ gives
\begin{align*}
& \mathcal L'(\rho,\mu) \zeta \\
& = \int_M \left( V + \frac 1{2 } \left\| -b + \frac 1{\rho} \left( J + \half \nabla \rho \right) \right\|^2 + \mu \right) \zeta +  \left \langle\frac 1{\rho} \left( J + \half \nabla \rho \right) - b,  \frac 1 {2 \rho} \nabla \zeta- \frac 1 {\rho^2} \left( J+ \half \nabla \rho \right) \zeta  \right \rangle \rho \ d x \\
& = \int_M \left(  V +  \mu + \half \|b\|^2 - \frac 1 {2 \rho^2} \|J\|^2 + \half \div b + \frac 1 {8 \rho^2} \|\nabla \rho\|^2 - \frac 1 {4 \rho} \Delta \rho \right) \zeta \ d x,
\end{align*}
where 
we used $\div J=0$.
We require that for any direction $\zeta$ the above expression equals zero, which is the case if and only if
\begin{equation} \label{eq:nonlinear_pde_fixed_current} - \frac 1 {4 \rho} \Delta \rho + \frac 1 {8 \rho^2} \|\nabla \rho\|^2 - \frac 1 {2 \rho^2} \|J\|^2 + W +  \mu = 0.\end{equation}
Note that we need to solve this equation for both $\rho$ and $\mu$, in combination with the constraint that $\rho > 0$ on $M$ and $\int_M \rho \ d x = 1$. 
By substituting $\phi = \sqrt{\rho}$, equation~\eqref{eq:nonlinear_pde_fixed_current} transforms into the equation~\eqref{eq:yermakov}.

We may then compute the cost corresponding to $\rho = \phi^2$ as
\begin{align}
\label{eq:value_fixed_current} 
\nonumber \mathcal C(\rho) & = \int_M \left( V + \frac 1{2 } \left\|-b + \frac 1{\rho} \left( J + \half \nabla  \rho \right) \right\|^2  \right) \rho \ d x \\
\nonumber & = \int_M \frac 1 { 4 } \Delta \rho - \frac 1 {8  \rho} \|\nabla \rho\|^2 + \frac 1 {2  \rho } \|J\|^2 - \frac {\rho} {2 } \|b\|^2 -  \frac {\rho} {2 } \div b - \mu \rho + \frac {\rho}{2 } \left\|-b + \frac 1{\rho} \left( J + \half \nabla \rho \right) \right\|^2  \ d x \\
& = \int_M \frac 1 { 4 } \Delta \rho - \frac {\rho}{2 } \div b - \mu \rho -  \langle b, J \rangle - \rho \left\langle b, \frac 1 {2 \rho} \nabla \rho \right\rangle + \frac 1 {\rho} \langle J, \half \nabla \rho \rangle \ d x \\
\nonumber & = - \left( \mu  + \int_M \langle b, J \rangle \ d x \right),
\end{align}
where we used~\eqref{eq:nonlinear_pde_fixed_current} in the first equality, and the following relations for the last equality:
\begin{align*}
\int_M \Delta \rho \ d x & = 0, \quad & \int_M \rho (\div b) \ d x & = - \int_M \langle \nabla \rho, b \rangle \ d x, \\
\int_M \frac 1 {\rho} \langle J, \nabla \rho \rangle \ d x & = - \int_M (\div J) \ln \rho  \ d x =  0, \quad & \int_M \rho \ d x  & = 1.
\end{align*}

We can only influence the first term in~\eqref{eq:value_fixed_current} by choosing $\rho$ or $\mu$, so we see that minimizing $\mathcal C$ therefore corresponds to finding the largest value of $\mu$ such that~\eqref{eq:nonlinear_pde_fixed_current}, or, equivalently,~\eqref{eq:yermakov}, admits a solution.

\subsection{Reversible solution -- stationary Schr\"odinger equation}
\label{sec:schrodinger}

In this section we consider the special case of the above problem for zero current density, $J = 0$. By Lemma~\ref{lem:reversible_equivalence}, this is equivalent to $u + b = - \frac 1 2 \nabla \Psi$ for some unknown $\Psi \in C^{\infty}(M)$, with $\rho = \exp(-\Psi)$. In other words, we demand the net force field (including the control) to be in gradient form, and the corresponding diffusion to be reversible; see Section~\ref{sec:reversible_solution}.

In this case~\eqref{eq:yermakov} transforms into the linear eigenvalue problem,
\begin{equation}
\label{eq:reversible_linear_pde} 
\half \Delta \phi - W \phi = \mu \phi.
\end{equation}
This is intriguing since this is in fact a time independent Schr\"odinger equation for the square root of a density function, analogous to quantum mechanics; even though our setting is entirely classical. See also \cite[equation (13)]{Chernyak2013} and the discussion following it.

By~\eqref{eq:value_fixed_current},
we are interested in the largest value of $\mu$ so that~\eqref{eq:reversible_linear_pde} has a solution $\phi$.
The optimal control field is then given by $u = \frac 1 {\phi} \nabla \phi - b$.

\begin{remark}
It is straightforward to check that if $b = - \half \nabla  U$ for some $U \in C^{\infty}(M)$, then $\phi = \exp(-\half U)$ satisfies~\eqref{eq:reversible_linear_pde} with $V = 0$, $\mu = 0$, resulting in $u = 0$. This corresponds to the intuition that, if $b$ is already a gradient, no further control is necessary to obtain a reversible invariant measure.
\end{remark}
\begin{remark}
We may also compare the case $b = -\half \nabla U$ with the result of Section~\ref{sec:reversible_solution}. There we obtained that, in case $A = 0$ and $b = -\half \nabla U$, the optimization problem for unconstrained $J$ resulted in a reversible solution. In other words, the constraint $J = 0$ does not need to be enforced, and the solution of this section should equal the solution obtained in Proposition~\ref{prop:reversible_solution}. Apparently, with $\psi$ as in Proposition~\ref{prop:reversible_solution}, we have that $\phi^2 = \psi^2 \exp(-U)$.
\end{remark}

\section{Discussion}
\label{sec:discussion}
In this paper we showed how stationary long term average control problems are related to eigenvalue problems (for the unconstrained problem and the problem constrained to a reversible solution, Sections~\ref{sec:hjb} and~\ref{sec:schrodinger}), elliptic PDEs (for the problem with fixed density, Section~\ref{sec:fixed_density}) or a non-linear eigenvalue problem (for the problem with fixed current density, Section~\ref{sec:fixed_current_density}). For this we fruitfully used the representation of an optimal control field $u$ in terms of the density function $\rho$ and the current density $J$.

The theory on existence of solutions and spectrum of operators is classical and we refer the interested reader to e.g. \cite{Evans2010, Taylor1996, Warner1983}. 

An interesting problem for further research is the following. One may ask the question whether we may obtain solutions when we constrain a certain flux $\int_M \langle A, J \rangle \ d x$ (see Section~\ref{sec:flux}) to a given value. In this case, one may use $\widetilde A = \mu A$ as a Lagrange multiplier and use the results of Section~\ref{sec:hjb} (for constrained flux) and Section~\ref{sec:fixed_density} (for constrained flux and density) to obtain necessary conditions. 

Apparently, in the case of a diffusion, density and current density completely characterize the ergodic behaviour. It seems an interesting problem to formulate a general theory which extracts from a given Markov processes the precise quantities which specify the long term behaviour of the process. 

\section*{Acknowledgements}
We kindly thank the referee for various important and useful suggestions for improvement and pointing out a few related references.

\appendix

\section{Flux}

\label{app:flux}
In this section we will give a natural interpretation of the term $\int_M \langle A,J \rangle \ d x$, namely as the flux of $J$ through a cross-section $\alpha$, or equivalently, the long term average intersection index of the stochastic process $(X(t))_{t \geq 0}$ with a cross-section. The section motivates the gauge field $A$ in the cost function. The remainder of this paper does not refer to this section.
For background reading in differential geometry, see \cite{Arnold1989, Warner1983}. See also \cite{Chernyak2009, Chernyak2013} where the considerations below are discussed from a physics perspective. 

Let $M$ be a compact, oriented, Riemannian manifold of dimension $n$.
Recall the notion of a \emph{singular $p$-chain in $M$ (with real coefficients)} as a finite linear combination $c = \sum a_i \sigma_i$ of smooth $p$-simplices $\sigma_i$ in $M$ where the $a_i$ are real numbers. 
Let $S_p(M;\R)$ denote the real vector space of singular $p$-chains in $M$. On $S_p(M;\R)$, $p \in \Z$, $p \geq 0$, are defined boundary operators $\partial_p : S_p(M;\R) \rightarrow S_{p-1}(M;\R)$. The \emph{$p$-th singular homology group of $M$ with real coefficients} is defined by 
\[ \mathcal H_p(M;\R) = \ker \partial_p / \im \partial_{p+1}.\] Elements of $\ker \partial_p$ are called $p$-cycles, and elements of $\im \partial_{p+1}$ are called $p$-boundaries.
The \emph{deRham cohomology classes} $\mathcal H^p_{\deRham}(M;\R)$ are defined, for $0 \leq p \leq m = \dim(M)$ as
\[ \mathcal H^p_{\deRham}(M) = \ker d_p / \im d_{p-1},\] where $d_p : \Lambda^p(M) \rightarrow \Lambda^{p+1}(M)$ denotes exterior differentiation.

Let $\alpha$ be a $p$-cycle. The functional $p_{\alpha} \in \Lambda^p(M) \rightarrow \int_{\alpha} \beta \in \R$ depends, by Stokes' theorem, only on the homology class $[\alpha]$ of $\alpha$, and the deRham cohomology class $[\beta]$ of $\beta$. An element $[\alpha] \in \mathcal H_p(M;\R)$ may therefore be considered an element of $\left(\mathcal H^p_{\deRham}(M)\right)^{\star}$. 
Now let $M$ be oriented and of dimension $n$. The mapping $q_{\beta}: [\gamma] \in \mathcal H^p_{\deRham}(M) \rightarrow \int_M \beta \wedge \gamma \in \R$ is an element of $\left(\mathcal H^p_{\deRham}(M)\right)^{\star}$. Poincar\'e duality states that the mapping $[\beta] \in \mathcal H^{m-p}_{\deRham}(M) \rightarrow q_{\beta} \in \left(\mathcal H^p_{\deRham}(M) 
\right)^{\star}$ is an isomorphism for compact $M$, i.e. $\mathcal H^{m-p}_{\deRham}(M)  \cong \left( \mathcal H^{p}_{\deRham}(M)\right)^{\star}$. Therefore, for compact, oriented $M$, we have
\[ \mathcal H_p(M;\R) \cong \left(\mathcal H^p_{\deRham}(M)\right)^{\star} \cong \mathcal H^{m-p}_{\deRham}(M).\]
Since an equivalence class in $\mathcal H^{m-p}_{\deRham}(M)$ has a unique harmonic representative, we conclude that for a $p$-cycle $\alpha$, there exists a unique harmonic $r_{\alpha} \in \Lambda^{m-p}(M)$ such that 
\[ \int_{\alpha} \beta = p_{\alpha}(\beta) = \int_M r_{\alpha} \wedge \beta, \quad [\beta] \in \mathcal H^p_{\deRham}(M).\]

In particular, for $p = m - 1$, we may interpret $\int_{\alpha} \star J$ (with $\star$ the \emph{Hodge star operator}) as the flux of $J$ through $\alpha$. This quantity may further be interpreted as the long term average intersection index of the stochastic trajectory $(X(t))_{t \geq 0}$ with respect to $\alpha$, i.e. the long term average of the number of intersections (with $\pm 1$ signs depending on the direction); see e.g. \cite[Section 0.4]{GriffithsHarris1994}. Specializing the above result to this situation, we obtain the following proposition.

\begin{proposition}
For a given $(n-1)$-cycle $\alpha$, there exists a unique harmonic $A \in \Lambda^1(M)$, which depends only on the singular homology class $[\alpha] \in \mathcal H_{m-1}(M;\R)$ of $\alpha$, such that $\int_{\alpha} \star J = \int_M \langle A, J \rangle \ d x$ for all $J \in \Lambda^1(M)$ satisfying $\delta J = 0$.
\end{proposition}

\section{Derivation of expression for long term average of current density}
\label{app:long_term_current}
In the physics literature (see e.g. \cite{Chernyak2009}), the current density is defined formally as
\begin{equation} \label{eq:J_physics} J^i_t(x) = \frac 1 t \int_0^t \dot X_s^i \delta(X_s - x) \ d s,\end{equation}
for $x \in M$, where $\delta$ is the Dirac delta function. 
We will derive an alternative expression for this quantity, using the model~\eqref{eq:controlled} for the dynamics.
Note that~\eqref{eq:J_physics} formally defines a vector field that acts on functions $f \in C^{\infty}(M)$ as
\begin{align*} J_t f(x) & = \frac 1 t \int_0^t \dot X_s^i \delta(X_s - x) \partial_i f(x) \ d s = \frac 1 t \int_0^t \dot X_s^i \delta(X_s - x) \partial_i f(X_s) \ d s = \frac 1 t \int_0^t \delta(X_s -x) \partial_i f(X_s) \circ d X_s^i.
\end{align*}
The $\delta$-function is still problematic. We may however formally compute the $L^2(M,g)$ inner product of the above expression with any $h \in C^{\infty}(M)$ with support in a coordinate neighbourhood $U$ containing $x$. This results in
\begin{align*}
\int_M h(x) J_t f(x) \ d x & = \frac 1 t \int_U h(x) \int_0^t \delta(X_s -x) \partial_i f(X_s) \circ d X_s^i \ d x = \frac 1 t \int_0^t h(X_s)  \partial_i f(X_s) \circ d X_s^i.
\end{align*}
Using the relation $Y \circ d Z = Y \ d Z + \half d [Y,Z]$ and~\eqref{eq:ito_equation}, we compute
\begin{align*} 
\int_U h(x) J_t f(x) \ d x & = \frac 1 t \int_0^t h(X_s)  \partial_i f(X_s) \left( \overline b_u^i (X_s)d s + \sigma_{\alpha}^i (X_s) \ d B_s^{\alpha} \right) + \half \frac 1 t \int_0^t \sigma_{\alpha}^i \sigma_{\alpha}^j  \partial_j \left(h \partial_i f \right)(X_s) \ d s \\
& \rightarrow \int_U \left\{ h(x)  (\partial_i f)(x) \left( \overline b_u^i (x)\right)  + \half g^{ij}  \partial_j \left(h \partial_i f \right)(x)\right\} \rho_u \ d x \quad (\mbox{almost surely as} \ t \rightarrow \infty) \\
& = \int_U  h(x) \left\{ \rho_u (\partial_i f) \left( \overline b_u^i \right)  - \half \frac 1 {\sqrt{|g|} } \left(  \partial_i f \right) \partial_j \left( \rho_u \sqrt{|g|} g^{ij}\right) \right\}(x) \ d  x,  \end{align*}
using Proposition~\ref{prop:long_term_density} and the law of large numbers for martingales \cite{NguyenPham1982}.
We find that the long term average vector field $J$ has components
\begin{align*} J^i & = \rho_u \overline b_u^i - \half \frac 1{\sqrt{|g|}} \partial_j \left( \rho_u \sqrt{|g|} g^{ij} \right)  = \rho_u \left( b_u^i + \half \sigma_{\alpha}^k \left( \partial_k \sigma_{\alpha}^i\right) \right) - \half g^{ij} \left(\partial_j \rho_u \right) - \half \rho_u \frac 1{\sqrt{|g|}} \partial_j \left( \sqrt{|g|} g^{ij} \right) \\
& = \rho_u b_u^i - \half g^{ij} \left( \partial_j \rho_u\right),
\end{align*}
where the last equality is a result of the identity
\[ \left(\nabla_{\sigma_{\alpha}} \sigma_{\alpha} \right)^i = \sigma_{\alpha}^k \left(\partial_k \sigma_{\alpha}^i \right) - \frac 1 {\sqrt{|g|}} \partial_j \left( \sqrt{|g|} g^{ij} \right),\]
which may be verified by straightforward calculation. 
%
%

\end{document}